\documentclass[12pt]{amsart}

\usepackage[latin1]{inputenc}
\usepackage{amsmath}
\usepackage{amssymb}
\usepackage{graphicx}
\usepackage{graphics}
\usepackage{amsthm}
\usepackage{amscd}
\usepackage{color}
\usepackage{amsfonts}
\usepackage{fancyhdr}
\newtheorem{theorem}{Theorem}[section]
\newtheorem{mainthm}{Theorem}
\newtheorem*{theorem*}{Theorem}
\newtheorem{corollary}[theorem]{Corollary}
\newtheorem{proposition}[theorem]{Proposition}
\newtheorem{lemma}[theorem]{Lemma}
\newtheorem{Remark}[theorem]{Remark}
\newtheorem*{definition*}{Definition}

\newtheorem{claim}[theorem]{Claim}
\newtheorem{definition}[theorem]{Definition}
\newtheorem*{Question*}{Question}

\def\N{\mathbb{N}}
\def\Z{\mathbb{Z}}

\def\R{\mathbb{R}}

\def\T{\mathbb{T}}

\def\norm #1{\Vert \,#1\, \Vert\,}
\makeatletter

\newcommand{\Rmnum}[1]{\expandafter\@slowromancap\romannumeral #1@}
\makeatother

\def\R{\mathbb{R}}
\def\S{\mathbb{S}}

\def\calE{\mathcal{E}}

\def\calF{\mathcal{F}}
\def\calG{\mathcal{G}}
\def\T{\mathbb{T}}
\def\ud{\mathrm{d}}
\def \diff {\operatorname{Diff}}
\def \home {\operatorname{Home}}

 \def\NN{{\mathbb N}}  
 \def\RR{{\mathbb R}} \def\SS{{\mathbb S}} \def\TT{{\mathbb T}}
   
 \def\ZZ{{\mathbb Z}}

\def\Si{\Sigma}

\def\cA{\mathcal{A}}  \def\cG{\mathcal{G}}  
  \def\cH{\mathcal{H}}  
  \def\cI{\mathcal{I}}  
   \def\cP{\mathcal{P}} 
\def\cE{\mathcal{E}}    \def\cW{\mathcal{W}}
\def\cF{\mathcal{F}}   \def\cR{\mathcal{R}}

\vspace{-2cm}
\title[Transverse foliations and partially hyperbolic diffeomorphisms]{Transverse foliations on the torus  $\T^2$ and partially hyperbolic diffeomorphisms on $3$-manifolds.}
\author{Christian Bonatti and  Jinhua ZHANG}
\vspace{-2cm}

\begin{document}
\maketitle
\begin{abstract} In this paper, we prove that given two $C^1$ foliations  $\mathcal{F}$ and $\mathcal{G}$ on $\mathbb{T}^2$ which
 are transverse, there exists a non-null homotopic loop $\{\Phi_t\}_{t\in[0,1]}$ in $\diff^{1}(\T^2)$ such that
  $\Phi_t(\calF)\pitchfork \calG$ for every $t\in[0,1]$, and $\Phi_0=\Phi_1= Id$.

  As a direct consequence, we get a general process for building new partially hyperbolic diffeomorphisms on closed $3$-manifolds.
   \cite{BPP} built a new example of  dynamically coherent non-transitive partially hyperbolic diffeomorphism on a closed $3$-manifold;
   the example in \cite{BPP} is obtained
  by composing the time $t$ map, $t>0$ large enough, of  a very specific  non-transitive Anosov  flow by a Dehn twist along a transverse torus.
  Our result shows that the same construction holds starting with any non-transitive Anosov flow on an oriented $3$-manifold. Moreover, for a given transverse torus,
  our result explains which type of  Dehn twists lead to partially hyperbolic diffeomorphisms.
\end{abstract}

\section{Introduction and statement of the main results}
The main motivation of this paper is the construction of new examples of partially hyperbolic diffeomorphisms on closed $3$-manifolds, initiated in \cite{BPP}. More precisely, our main
result is
a topological result which was missing for \cite{BPP} getting a general construction instead of a precise example.  Nevertheless, this topological result deals with very
elementary objects and is interesting by itself. We first present it below independently from its application on partially hyperbolic diffeomorphisms.

\subsection{Pair of transverse foliations on $\TT^2$}
 The space of $1$-dimensional (non-singular)  smooth foliations on the torus $\T^2$ has several connected components which are easy to describe: such a foliation is directed by  a smooth unit vector
 field, which can be seen has a map $X\colon \TT^2\to \SS^1$;  such a map induces a morphism $X_*\colon \pi_1(\TT^2)=\ZZ^2\to \ZZ$ and two foliations can be joined by a path of non-singular foliations
 if and only if the induced morphisms coincide. The group $\diff_0(\TT^2)$ of diffeomorphisms of $\TT^2$ isotopic to the identity map has a natural action on the space of foliations.

 In this paper,  we consider pairs $(\cF,\cG)$ of transverse foliations on $\TT^2$.  For any such a pair $(\cF,\cG)$of transverse foliations,  we consider the open subset of $\diff_0(\TT^2)$ of all diffeomorphism $\varphi$ so that
 $\varphi(\cF)$ is transverse to $\cG$. Our main result below shows  that this open subset contains non-trivial loops.

\begin{mainthm}\label{t.transverse} Let $\mathcal{F}$ and $\mathcal{G}$ be two $C^1$ one-dimensional foliations on $\mathbb{T}^2$ and they are transverse.
Then there exists a continuous family $\{\Phi_t\}_{t\in[0,1]}$ of $C^1$ diffeomorphisms on $\mathbb{T}^2$ such that
\begin{itemize} \item  $\Phi_0=\Phi_1=Id$;
\item For every $t\in[\,0,1\,]$, the $C^1$ foliation  $\Phi_t(\mathcal{F})$ is transverse to $\mathcal{G}$;
\item For every point $x\in\mathbb{T}^2$, the closed curve $\Phi_t(x)$ is non-null homotopic.
\end{itemize}
\end{mainthm}

Our main theorem is implied by the following two theorems, according to the two cases described in Definition~\ref{d.parallel} below:

\begin{definition}\label{d.parallel}We say that two foliations $\cF$ and $\cG$ of the torus $\TT^2$
\emph{have parallel compact leaves} if and only if there exist a compact leaf of $\cF$ and a compact leaf of $\cG$ which
are in the same free homotopy class.

Otherwise, we say that $\cF$ and $\cG$ \emph{have no parallel compact leaves} or that they are \emph{without parallel compact leaves}.

\end{definition}

\begin{theorem} \label{thm.nonparallel} Let $\mathcal{F}$ and $\mathcal{G}$ be two $C^1$ one-dimensional transverse  foliations on $\mathbb{T}^2$, without parallel compact leaves.
Then for any $\alpha\in \pi_{1}(\T^2)$, there exists a continuous family $\{\Phi_t\}_{t\in[0,1]}$ of $C^1$
  diffeomorphisms on $\mathbb{T}^2$ such that
\begin{itemize} \item  $\Phi_0=\Phi_1=Id$;
\item For every $t\in[\,0,1\,]$, the $C^1$ foliation $\Phi_t(\mathcal{F})$ is transverse to $\mathcal{G}$;
\item For every point $x\in\mathbb{T}^2$, the closed curve $\Phi_t(x)$ is in the homotopy class of $\alpha$.
\end{itemize}
\end{theorem}

The proof  of Theorem \ref{thm.nonparallel} consists in endowing $\TT^2$ with coordinates in which the foliations $\cF$ and $\cG$ are separated by $2$ affine foliations (i.e. $\cF$ and $\cG$ are tangent to two transverse
constant cones). Thus in these coordinates every translation leaves $\cF$ transverse to $\cG$, concluding.

\begin{theorem}\label{thm.parallel} Let $\mathcal{F}$ and $\mathcal{G}$ be two $C^1$ one-dimensional foliations on $\mathbb{T}^2$ and
they are transverse. Assume that $\cF$ and $\cG$ have parallel compact leaves  which are in the homotopy class  $\alpha\in \pi_1(\TT^2)$.
Then, for each $\beta\in \pi_1(\TT^2)$, one has that $\beta\in \langle\alpha\rangle$ if and only if there exists a continuous family $\{\Phi_t\}_{t\in[0,1]}$ of $C^1$
 diffeomorphisms on $\mathbb{T}^2$ such that
\begin{itemize} \item  $\Phi_0=\Phi_1=Id$;
\item  For every $t\in[\,0,1\,]$, the $C^1$ foliation $\Phi_t(\mathcal{F})$ is transverse to $\mathcal{G}$;
\item For every point $x\in\mathbb{T}^2$, the closed curve $\Phi_t(x)$ is in the homotopy class of $\beta$.
\end{itemize}
\end{theorem}

One easily checks that, if $\cF$ and $\cG$ are transverse $C^1$ foliations having compact leaves in the same homotopy class, then every compact leaf $L_{_\cF}$ of $\cF$ is disjoint from
every compact leaf $L_{_\cG}$ of $\cG$. If $\{\Phi_t\}_{t\in[0,1]}$ is an isotopy so that $\Phi_0$ is the identity map and $\Phi_t(\cF)$ is transverse to $\cG$, then $\Phi_t(L_{_\cF})$ remains
disjoint from $L_{_\cG}$: this implies the \emph{if} part of Theorem~\ref{thm.parallel}.  The \emph{only if} part will be the aim of Section~\ref{s.parallel}.
The proof consists in endowing $\TT^2$ with coordinates in which $\TT^2$ is divided into vertical adjacent annuli in which the foliations are separated by affine foliations:
now the vertical translations preserve the vertical annuli and map $\cF$ on  foliations transverse to $\cG$.

\begin{Remark}As the transversality of foliations is an open condition, any loop of diffeomorphisms $C^1$-close to the loop $\{\Phi_t\}_{ t\in[0,1]}$ has the announced properties.
Therefore, in Theorems~\ref{t.transverse}, \ref{thm.nonparallel}, and \ref{thm.parallel},
one can choose the loop $t\mapsto \Phi_t$  so that the map   $(t,x)\mapsto \Phi_t(x)$, for $(t,x)\in \SS^1\times \TT^2$, is   smooth.

\end{Remark}

\begin{definition}\label{d.G}  Let $(\cF,\cG)$ be a pair of transverse foliations of $\TT^2$.
We denote by $G_{\cF,\cG}\subset \pi_1(\TT^2)$ the group  defined as follows:
\begin{itemize}
 \item if $\cF$ and $\cG$ have no parallel compact leaves then $G_{\cF,\cG}=\ZZ^2=\pi_1(\TT^2)$;
 \item if $\cF$ and $\cG$ have parallel compact leaves, let $\alpha\in\pi_1(\ZZ)$ be the homotopy class of these leaves.  Then
 $G_{\cF,\cG}=\langle \alpha\rangle =\ZZ\cdot\alpha\subset \pi_1(\TT^2)$.
\end{itemize}
\end{definition}

\subsection{Dehn twists and pairs of transverse $2$-foliations on $3$-manifolds}

The aim of this paper is to build partially hyperbolic diffeomorphisms on $3$-manifold by composing the time $t$-map of an Anosov flow by  Dehn-twists along a transverse tori.
In this section we define the notion of Dehn twists, and we state a straightforward consequence of Theorems~\ref{thm.nonparallel} and \ref{thm.parallel} producing Dehn twists preserving the transversality
of two $2$-dimensional foliations.

\begin{definition} Let $u=(n,m)\in \ZZ^2=\pi_1(\TT^2)$. A diffeomorphism $\psi\colon [0,1]\times \TT^2\to [0,1]\times  \TT^2$ is called a
\emph{Dehn twist of $[0,1]\times \TT^2$ directed by $u$} if:
\begin{itemize}
 \item $\psi$ is of the form $(t,x)\mapsto (t,\psi_t(x))$,  where $\psi_t$ is a diffeomorphism of $\TT^2$ depending smoothly on $t$.
 \item $\psi_t$ is the identity map for $t$ close to $0$ or close to $1$.
 \item the closed path $t\mapsto \psi_t(O)$ on $\TT^2$  is freely homotopic to $u$ (where $O=(0,0)$ in $\TT^2=\RR^2/\ZZ^2$).
\end{itemize}

\end{definition}

\begin{definition}Let $M$ be an oriented $3$-manifold and let $T\colon \TT^2\hookrightarrow M$ be an embedded torus. Fix $u\in \pi_1(T)$.  We say that a diffeomorphism
$\psi\colon M\to M$ is \emph{a Dehn twist along $T$ directed by $u$} if there is an orientation preserving diffeomorphism $\varphi\colon [0,1]\times \TT^2\hookrightarrow M$ whose restriction
to $\{0\}\times \TT^2$ induces $T$, and so that:
\begin{itemize}
 \item $\psi$ is the identity map out of $\varphi([0,1]\times \TT^2)$. In particular, $\psi$ leaves invariant $\varphi([0,1]\times \TT^2)$;
 \item The diffeomorphism $\varphi^{-1}\circ \psi \circ \varphi\colon [0,1]\times \TT^2\to [0,1]\times \TT^2$ is a Dehn twist directed by $u$.
\end{itemize}

\end{definition}

\begin{proposition}\label{p.foliations} Let $\cF, \cG$ be a pair of $2$-dimensional foliations on a  $3$-manifold $M$, and let $\cE$ be the $1$-dimensional foliation obtained as $\cE=\cF\cap \cG$.
Assume that there is an embedded torus $T\subset M$ which is transverse to $\cE$ (hence $T$ is transverse to $\cF$ and $\cG$). We denote by $\cF_T,\cG_T$ the $1$-dimensional foliations on $T$ obtained
as intersection of $\cF$ and $\cG$ with $T$, respectively.

Then for every $u\in G_{_{\cF_T,\cG_T}}\subset \pi_1(T)$,   there is a Dehn twist $\psi$ along $T$ directed by $u$ so that $\psi(\cF)$ is transverse to $\cG$.
\end{proposition}

\subsection{Building partially hyperbolic diffeomorphisms on $3$-manifolds}

In order to state our main result,  we first need to define the notions of \emph{partially hyperbolic diffeomorphism} and  of  \emph{Anosov flow}.

\subsubsection{Definition of partially hyperbolic diffeomorphisms}

A diffeomorphism $f$ of a Riemannian closed $3$-manifold $M$ is called \emph{ partially hyperbolic} if there is a $Df$-invariant splitting $TM= E^s\oplus E^c\oplus E^u$
in direct sum of $1$-dimensional bundles so that

\begin{enumerate}\item[{\bf (P1):}~]there is an integer $N>0$ such that  for any $x\in M$
and any unit vectors $u\in E^s(x)$, $v\in E^c(x)$  and $w\in E^u(x)$,
 one has:
 $$\|Df^N(u)|<\inf\{ 1, \|Df^N(v)\|\}\leq \sup\{ 1,\|Df^N(v)\|\} <\|Df^N(w)\|.$$
\end{enumerate}

A diffeomorphism $f$ of a Riemannian closed $3$-manifold $M$ is called \emph{absolute partially hyperbolic} if it is partially hyperbolic satisfying the stronger assumption
\begin{enumerate}\item[\bf{(P2):}] there are $0<\lambda<1<\sigma $ and an integer $N>0$ so that for any $x,y,z\in M$
and any unit vectors $u\in E^s(x)$, $v\in E^c(y)$ and $w\in E^u(z)$,
 one has:
 $$\|Df^N(u)|<\lambda<\|Df^N(v)\|<\sigma<\|Df^N(w)\|.$$
\end{enumerate}

We refer the reader to \cite[Appendix B]{BDV} for the first elementary properties
and \cite{HHU1} for a
survey book of results and questions for partially hyperbolic diffeomorphisms.

\subsubsection{Definition of Anosov flows}

A vector field $X$ on a $3$-manifold $M$ is called an \emph{Anosov vector field}  if there is a splitting  $TM= E^s\oplus \RR\cdot X\oplus E^u$ as a direct sum of $1$-dimensional bundles which are
invariant by the flow $\{X_t\}_{t\in\RR}$ of $X$, and so that the vectors in $E^s$ are uniformly contracted and the vectors in $E^u$ are uniformly expanded by the flow of $X$.

Notice that $X$ is Anosov if and only if $X$ has no zeros and if there is $t>0$ so that $X_t$ is partially hyperbolic.

The bundles $E^{cs}= E^s\oplus \RR\cdot X$ and $E^{cu}=\RR\cdot X\oplus E^u$ are called the weak stable and unstable  bundles (respectively).
They are tangent to transverse $2$-dimensional foliations denoted by $\cF^{cs}$ and $\cF^{cu}$ respectively,
which are of class $C^1$ if $X$ is of class at least $C^2$.

The bundles $E^s$ and $E^u$ are called the strong stable and strong unstable bundles,  and are tangent to $1$-dimensional foliations denoted by $\cF^{ss}$ and $\cF^{uu}$ which are  called the strong stable and the strong unstable foliations, respectively.

Notice that being an Anosov vector field is an open condition in the set of $C^1$-vector fields and that the structural stability implies that all the flows
$C^1$-close to an Anosov flow are Anosov flows topologically equivalent to it.  Therefore, for our purpose here we may always   assume, and we do it,
that the Anosov flows we consider are smooth.

The most classical Anosov flows on $3$-manifolds are the geodesic flows of hyperbolic closed surfaces and the suspension of hyperbolic linear automorphisms of $\TT^2$
(i.e. induced by an hyperbolic element of $SL(2,\ZZ)$). In 1979, \cite{FW} built the first example of a non-transitive Anosov flow on a closed $3$-manifold.  Many other examples of
transitive or non-transitive Anosov flows have been built in \cite{BL,BBY}.

\subsubsection{Transverse tori}

If $X$ is an Anosov vector field  on an oriented closed $3$-manifold $M$ and if $S\subset M$ is an immersed closed surface  which is transverse to $X$ then
\begin{itemize}
 \item $S$ is oriented (as transversely oriented by $X$);
 \item $S$ is transverse to the weak foliations $\cF^{cs}$ and $\cF^{cu}$ of $X$ and these foliations induce on $S$ two $1$-dimensional $C^1$-foliations $\cF^s_S$ and
 $\cF^u_S$, respectively, which are transverse.
 \item as a consequence of the two previous items, $S$ is a torus.
\end{itemize}

A \emph{transverse torus} is an embedded torus $T\colon \TT^2\hookrightarrow M$ transverse to $X$ and we denote by $\cF^s_T$ and $\cF^u_T$ the $1$-dimensional $C^1$  foliations induced on
$T$
obtained  by intersections of  $T$ with $\cF^{cs}$ and   with  $\cF^{cu}$, respectively.  These foliations are transverse.  Therefore Theorem~\ref{t.transverse} associates to $(\cF^s_T,\cF^u_T)$
 and  a  subgroup $G_{_{\cF^s_T,\cF^u_T}}$ of $\pi_1(T)$  which is either a cyclic group if $\cF^s_T$ and $\cF^u_T$ have parallel compact leaves  or the whole $\pi_1(T)$ otherwise.

Let $T_1,\dots, T_k$ be a finite family  of  transverse tori.  We say that \emph{$X$ has no return on $\bigcup_i T_i$} if  each torus $T_i$ is an embedded torus, the $\{T_i\}$ are pairwise disjoint and
each orbit of $X$ intersects $\bigcup_i T_i$
in at most $1$ point.

 \vskip 2mm

A \emph{Lyapunov function} for $X$ is a function which is not increasing along every orbit, and which is strictly decreasing along every orbit which is not chain recurrent.

In \cite{Br} Marco Brunella noticed that a non-transitive Anosov vector field  $X$  on an  oriented closed $3$-manifold $M$
always admits a  smooth Lyapunov function whose regular levels separate the \emph{basic pieces}
of the flow;  such a regular level is a disjoint union of transverse tori $T_1,\cdots ,T_k$.  One can check the following statement:

\begin{proposition}\label{p.Lyapunov function} Let $X$ be a (non-transitive) Anosov vector field on an oriented closed $3$-manifold $M$.  Then the two following assertions are equivalent:
\begin{enumerate}
 \item $T_1,\dots, T_k$ are transverse tori so that
 $X$ has no return on $\bigcup_i T_i$.
 \item there is a smooth Lyapunov function $\theta\colon M\to \RR$  of $X$ for which  the $T_i$, $i\in\{1,\dots, k\}$ are (distinct) connected components of the same regular level
 $\theta^{-1}(t)$ for some $t\in \RR$.
\end{enumerate}
\end{proposition}

We are now ready to state our main result.

\subsection{Statement of our main result}

\begin{mainthm}\label{t.Anosov} Let $X$ be a smooth (non-transitive) Anosov vector field  on  an  oriented closed $3$-manifold $M$,
and let $T_1,\dots, T_k$ be   transverse tori so that $X$ has no return on $\bigcup_iT_i$.
We endow each $T_i$ with the pair $(\cF^s_{i}, \cF^u_{i})$ of
transverse $1$-dimensional  $C^1$  foliations obtained by intersection of
$T_i$ with the weak stable and unstable (respectively) foliations of $X$;  let
$$G_i= G_{_{\cF^s_{i},\cF^u_{i}}}\subset \pi_1(T_i)$$
denote the subgroup associated  to the pair $(\cF^s_{i},\cF^u_{i})$ by Theorems~\ref{thm.nonparallel} and \ref{thm.parallel}.

Then for any family $u_1\in G_1, \dots, u_k\in G_k$ there is a family $\Psi_i$ of  Dehn twists along $T_i$ directed by $u_i$, and whose supports are pairwise disjoint
and there is $t>0$ so that the composition

$$f=\Psi_1\circ \Psi_2\circ\cdots\circ \Psi_k\circ X_t$$
is an absolute partially hyperbolic diffeomorphism of $M$.

Furthermore, $f$ is robustly dynamically coherent, the center stable foliation $\cF^{cs}_f$ and center unstable foliation $\cF^{cu}_f$ are \emph{plaque expansive}.
\end{mainthm}

In a forthcoming work with a different group of authors, one will remove the hypothesis that the $T_i$ are connected components of a regular level of a Lyapunov function.  We state
this result here with this restrictive hypothesis in order that this result is a straightforward consequence of  Theorems~\ref{thm.nonparallel} and \ref{thm.parallel}. Removing this
hypothesis will require further very different arguments.

\vskip 5mm
{\bf Acknowledgment: }
We would like to  thank Rafael Potrie  whose questions motivate this paper.

Jinhua Zhang  would like to thank China Scholarship Council (CSC) for financial support (201406010010).

\section{Preliminary: foliations on the torus and its classification}

In this section, we give the definitions and results we need.

We denote by $\SS^1$ the circle $\SS^1=\RR/\ZZ$ and by $\TT^2$ the torus $\RR^2/\ZZ^2=\SS^1\times \SS^1$.

\subsection{Complete transversal}

\begin{definition} Given a  $C^r$ $(r\geq 1)$ foliation $\mathcal{F}$ on   $\mathbb{T}^2$. We say that a $C^1$ simple closed curve
$\gamma$ is \emph{a complete transversal} or a \emph{complete transverse cross section}  of the foliation  $\mathcal{F}$, if $\gamma$ is transverse to $\mathcal{F}$ and every leaf of $\mathcal{F}$
 intersects $\gamma$.
\end{definition}

\begin{lemma}\label{compact leaf} Consider a $C^1$ foliation $\calE$ on $\T^2$. Assume that there is a simple smooth closed
 curve $\gamma$ which is transverse to $\calE$ and is not a complete transversal of $\calE$. Then there exists a compact leaf of
 $\calE$ which is in the homotopy class of $\gamma$.
\end{lemma}
This lemma is classical.  As the proof is short, we include it for completeness.
\begin{proof} Cut the torus along $\gamma$: we get a cylinder $C$ endowed with a foliation transverse to its boundary. Furthermore,
by assumption, this foliation admits a leaf which remains at a uniform distance away from the
 boundary of $C$. Hence, the closure of that leaf is also far from the boundary of $C$. By Poincare-Bendixion theorem, the closure of this leaf contains a compact
 leaf, thus this compact leaf is disjoint from the boundary   $C$. Furthermore, as the foliation is not singular, this leaf cannot be homotopic to $0$ in  the annulus, hence
 it is homotopic to the boundary components.
 \end{proof}
\vspace{9mm}
  \begin{figure}[h]
\begin{center}
\def\svgwidth{0.2\columnwidth}
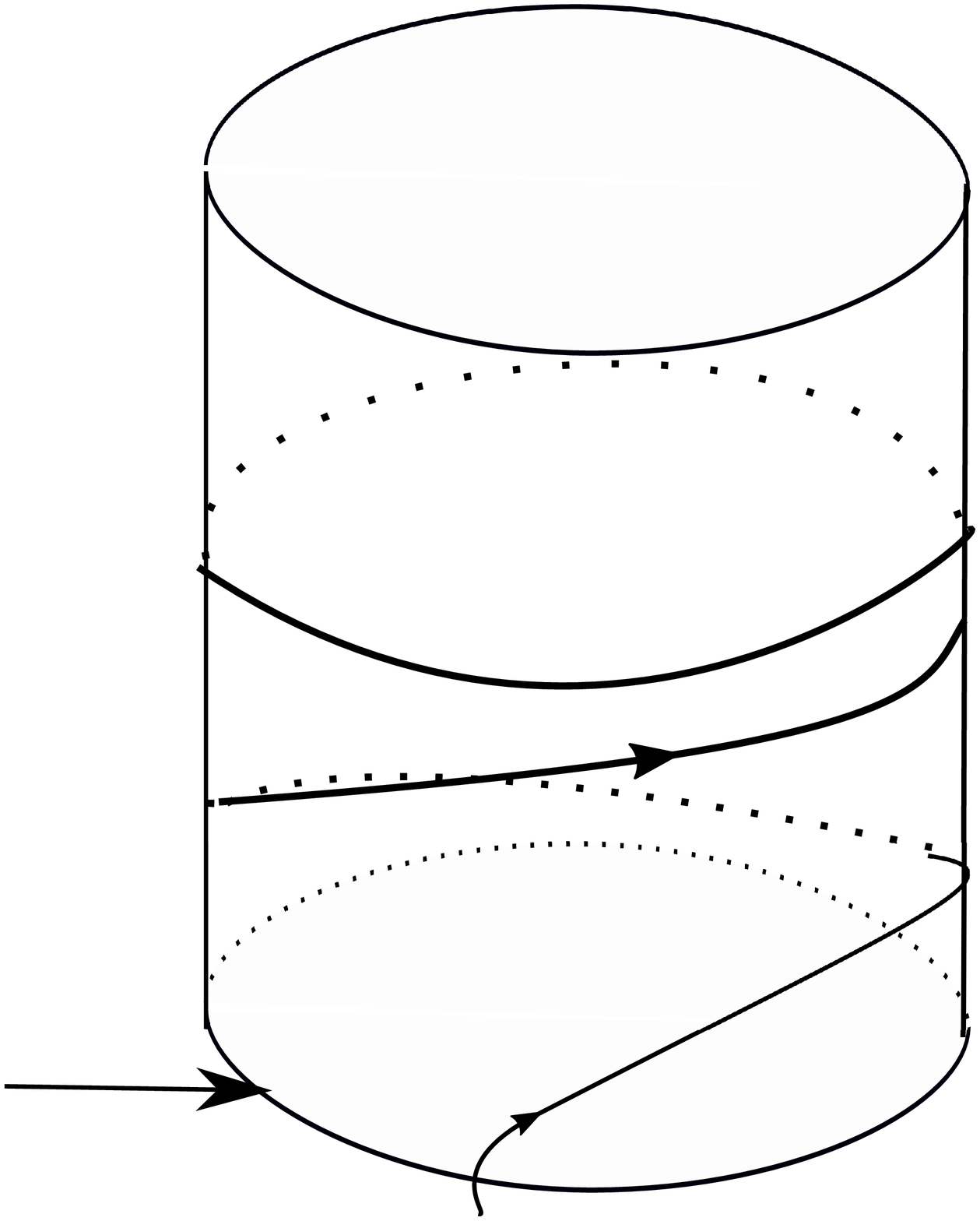
  \caption{}
\end{center}
\end{figure}

\begin{definition} Given two $C^r$ foliations $\calF$ and $\calF^{\prime}$ on manifolds $M$ and $M'$, respectively. $\calF$ and $\calF^{\prime}$ are \emph{$C^r$ conjugate}
if there exists a $C^r$ diffeomorphism $f\colon M\to M'$ such that $f(\calF)=\calF^{\prime}$.
\end{definition}


 \subsection{Reeb components}

 \begin{definition}[Reeb component] Given a foliation $\calF$ on   $\T^2$, we say that $\calF$ has a \emph{Reeb component}, if there
 exists a compact annulus $A$ such that
 \begin{itemize}
 \item  the boundary $\partial{A}$ is the union of two compact leaves of $\calF$;
 \item there is no compact leaf in the interior of $A$;
 \item first item above implies that  $\calF$ is orientable restricted to $A$, so let us choose an orientation. We require that the two oriented compact
 leaves  are in opposite homotopy classes.
 \end{itemize}
 \end{definition}

\begin{figure}[h]
\begin{center}
\def\svgwidth{0.5\columnwidth}
  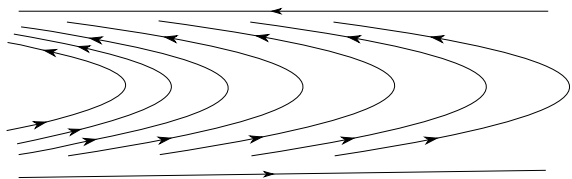
  \caption{Reeb component}
\end{center}
\end{figure}

By using the Poincar\'e-Bendixson theorem, one easily checks the following classical result:
\begin{proposition}\label{p.Reeb}
Let $\cF$ and $\cG$ be two transverse foliations on   $\TT^2$.  Assume that $\cF$ admits a Reeb component $A$.  Then $\cG$
admits a compact leaf contained in the interior of $A$.  Thus $\cF$ and $\cG$ have parallel compact leaves.
\end{proposition}

We state now a classification theorem  which can be found in \cite{HH}:
\begin{theorem}\label{thm.classification} \cite[Proposition 4.3.2]{HH} For any $C^r$ foliation $\mathcal{F}$ on $\mathbb{T}^2$,
 we have the following:
\begin{itemize}\item[--] Either $\mathcal{F}$ has Reeb component;
\item [--]or $\mathcal{F}$ is $C^r$ conjugated to the suspension of a $C^r$ diffeomorphism on $S^1$.
\end{itemize}
\end{theorem}

In general the union of the compact leaves of a foliation may fail to  be compact. But, for codimension $1$ foliations we have  the following theorem due to
A.Haefliger.
\begin{theorem}\label{H} \cite{H} For any $C^r$ $(r\geq 1)$ codimension one foliation $\calF$ on a compact manifold $M$, the set
 $$\{x\in M| \textrm{The $\calF$-leaf through $x$ is compact}\}$$ is a compact subset of $M$.
\end{theorem}

\subsection{Translation and rotation numbers}

In this section we recall very classical Poincar\'e theory on the rotation number of a circle homeomorphism. We refer to \cite{HK} for more details.

We denote by $\widetilde{homeo}^{+}(\mathbb{R})$ the set of orientation preserving homeomorphisms on $\mathbb{R}$ which commute
 with the translation  $t\mapsto t+1$. Recall that the elements of $\widetilde{homeo}^{+}(\mathbb{R})$ are precisely the lifts on $\RR$ of the
 orientation preserving homeomorphisms of $\SS^1$.

Let $H\in\widetilde{homeo}^{+}(\mathbb{R})$ be  the lift of $h\in Homeo^+(\SS^1)$. Poincar\'e noticed that the ratio $\frac{H^{n}(x)-x}{n}$ converges uniformly, as $n\to\pm\infty$, to some constant $\tau(H)$ called
the \emph{translation number  of $H$}. The projection $\rho(h)$ of $\tau(H)$ on $\RR/\ZZ$ does not depend on the lift $H$ and is called \emph{the rotation number of $h$}.

We can find  the following observations in many books, in particular in \cite{HK}.

\begin{Remark} The rotation number is
rational if and only if $h$ admits a periodic point.
\end{Remark}

\begin{proposition}\label{p.continuity of translation number}\cite[Proposition 11.1.6]{HK} $\tau(\cdot)$ varies continuous in
$C^0$-topology.
\end{proposition}

\begin{proposition}\label{p.comparing translation number}\cite[Proposition 11.1.9]{HK} Let $H,\,F\in\home^{+}_{0}(\RR)$.
Assume that $\tau(H)$ is irrational and   $H(x)<F(x)$, for any $x\in\RR$.
Then we have that $\tau(H)<\tau(F)$.
\end{proposition}

Poincar\'e theory proves that a homeomorphism $h$ of the circle with irrational rotation number is semi-conjugated to the rotation $R_{\rho(h)}$;  but  $h$ may fail to be conjugated to $R_{\rho(h)}$
even if $h$ is a $C^1$-diffeomorphism (Denjoy counter examples). However the semi-conjugacy is a $C^0$-conjugacy if $h$ is a $C^2$-diffeomorphism (Denjoy theorem). In general,
if $h$ is a smooth circle diffeomorphism with irrational rotation number, the conjugacy to the corresponding rotation may fail to be a diffeomorphism. However Herman proved that there are
generic conditions on $\rho(h)$ ensuring the smoothness of the conjugacy:

\begin{theorem}\label{He} \cite{He} Let $f\in \diff^{r}(\mathbb{S}^1)$ ($r\geq 3$) be a diffeomorphism of the circle. If the rotation number of $f$ is diophantine,
then $f$ is $C^{r-2}$ conjugated to an irrational rotation.
\end{theorem}

\subsection{Foliations without  compact leaves on the annulus}

Let $\cF$ be a $C^r$ foliation on the annulus $\SS^1\times [0,1]$ so that
\begin{itemize}
 \item $\cF$ is transverse to the boundary $\SS^1\times \{0,1\}$;
 \item $\cF$ has no compact leaves in the annulus.
\end{itemize}
Thus Poincar\'e-Bendixson theorem implies that every leaf entering through $\SS^1\times \{0\}$ goes out through a point of $\SS^1\times \{1\}$.
The map $\cP_\cF\colon \SS^1\times\{0\}\to\SS^1\times\{1\}$,  which associates an entrance point  $(x,0)$ of a leaf in the annulus to  its exit point at $\S^1\times\{1\}$,
is called the holonomy of $\cF$.

Consider the universal cover $\RR\times [0,1]\to\SS^1\times [0,1]$; we denote by $\tilde \cF$ the lift of $\cF$ on $\RR\times [0,1]$ and by $\widetilde{\cP_\cF}$ the holonomy of
$\tilde \cF$.  Note that $\widetilde{\cP_\cF}$ is a lift of $\cP_\cF$.

We will use the following classical and elementary results:

\begin{proposition}\label{p.holonomies} Let $\cF,\cG$ be $C^r$-foliations, $r\geq 0$, on the annulus $\SS^1\times [0,1]$ so that :
\begin{itemize}
\item The foliations $\cF$ and $\cG$ are transverse to the boundary $\SS^1\times\{0,1\}$ and have no compact leaves in the annulus;
 \item The foliations $\cF$ and $\cG$ coincide in a neighborhood of the boundary  $\SS^1\times \{0,1\}$;
 \item The foliations $\cF$ and $\cG$ have same holonomy, that is $\cP_\cF=\cP_\cG$.
\end{itemize}
Then there is a $C^r$ diffeomorphism $\varphi\colon \SS^1\times[0,1]\to\SS^1\times [0,1]$ which coincides with the identity map in a neighborhood of the boundary $\SS^1\times\{0,1\}$ and
so that $$\varphi(\cG)=\cF.$$

If furthermore the lifted foliations $\tilde \cF$ and $\tilde \cG$ have same holonomies, that is  $\widetilde{\cP_\cF}=\widetilde{\cP_\cF}$,  then $\varphi$ is isotopic (relative  to a neighborhood of the boundary)
to the identity map.
\end{proposition}

An important step for proving Proposition~\ref{p.holonomies} is the next classical result that we will also use several times:

\begin{proposition}\label{p.suspension} Let $\cF$ be a $C^r$ ($r\geq 1$) foliation  on the annulus $\SS^1\times[0,1]$,  transverse to the boundary and without compact leaf.
Then there is a smooth surjection $\theta\colon \SS^1\times[0,1]\to [0,1]$ so that $\cF$ is transverse to the fibers of  $\theta$.
\end{proposition}
A classical consequence of Proposition~\ref{p.suspension} is that a foliation on   $\TT^2$ admitting a complete transversal is
conjugated to the suspension of the first return map on this transversal.

\section{Existence of a complete transversal curve for two transverse foliations without  parallel compact leaves}\label{s.section}
In this section, we consider two foliations $\cF$, $\cG$ on the torus $\TT^2$ which do not have parallel compact leaves (see Definition~\ref{d.parallel}).
According to Proposition~\ref{p.Reeb} the foliation $\cF$ and $\cG$ have no Reeb component. In particular, $\cF$ and $\cG$ are orientable.
By Theorem \ref{thm.classification}, for
each of them, there exist complete transverse cross sections. In this section, we prove that any two transverse foliations  without parallel compact leaves
share a complete transverse cross section.

\begin{proposition} \label{st} If two $C^1$ foliations  $\mathcal{F}$ and $\mathcal{G}$ are transverse on $\T^2$ and have no parallel compact leaves, then there
exists a smooth simple closed curve $\gamma$ which is a complete transversal to both $\mathcal{F}$ and $\mathcal{G}$.
\end{proposition}
\begin{proof}

As noticed before the statement of Proposition~\ref{st}, the foliations $\mathcal{F}$ and $\mathcal{G}$ have no Reeb component and therefore  $\mathcal{F}$ and $\mathcal{G}$ are
  orientable. Thus there exist two unit vector fields $X,Y$ such that $X$ and $Y$ are
  tangent to the foliations $\mathcal{F}$ and $\mathcal{G}$ respectively.

  Since $X$ and $Y$  are transverse, the vector field $\frac{1}{2}\,X+\frac{1}{2}\,Y$ is
  transverse to both $\mathcal{F}$ and $\mathcal{G}$. Let $Z$ be a smooth vector field $C^0$ close enough to
   $\frac{1}{2}\,X+\frac{1}{2}\,Y$ so that $Z$ is non-singular  and transverse to both foliations $\mathcal{F}$ and $\mathcal{G}$.
Furthermore, up to perform a small perturbation, we can assume that $Z$ admits a periodic orbit $\tilde \gamma$
which is a simple closed curve transverse to   both $\cF$ and $\cG$.


According to Lemma~\ref{compact leaf}, if $\tilde\gamma$ is not a complete transversal of one of the foliations $\cF$ or $\cG$,
this foliation admits a compact leaf homotopic to $\tilde \gamma$.  As $\cF$ and $\cG$ have no parallel compact leaves,
this may happen to at most one of $\cF$ and $\cG$. In other words, $\tilde \gamma$ is a complete transversal for at least one of the
foliations, thus we assume that $\tilde \gamma$ is a complete transversal for $\cF$. If $\tilde \gamma$ is a complete transversal
for $\cG$,  we are done.

Thus we assume that  it is not the case. Therefore Lemma~\ref{compact leaf} implies that $\cG$ has compact
leaves which  are in the homotopy class of $\tilde \gamma$. We denote by $C_{_\cG}$ a compact leaf of $\cG$, and  we denote by   $L$ a segment of leaf of  $\mathcal{F}$ with
endpoints $p,q$ on $\tilde \gamma$ and whose interior is disjoint from $\tilde\gamma$; furthermore,   if
   $\mathcal{F}$ has a compact leaf, we choose $L$ contained in a compact leaf of $\mathcal{F}$. We  denote by  $\sigma\subset \tilde\gamma$ the (unique) non trivial oriented segment so that
   \begin{itemize}
    \item $\sigma$ joins the final point $q$ of $L$ to
   its initial point $p$;
   \item  the interior of $\sigma$ is disjoint from  $\{p,q\}$;
   \item the orientation of $\sigma$ coincides with the transverse orientation of the foliation $\cG$,  given by the vector field $X$ directing $\cF$.
   \end{itemize}
 Thus the concatenation $\gamma_0= L\cdot \sigma$ is a closed curve (which is simple unless $p=q$ in that case $p=q$ is the unique and non topologically transverse intersection point)
 consisting in one leaf segment  and one transverse segment to $\cF$.
 A classical process allows us to smooth $\gamma_0$ into a smooth curve $\gamma$ transverse to $\cF$ (see Figure 3 for the case $p\neq q$ and Figure 4 for the case $p=q$), and the choice of the oriented segment $\sigma$ allows us to choose $\gamma$ transverse to $\cG$.
 Furthermore, we have that
 \begin{itemize}
 \item $\gamma$ cuts the compact leaf $C_{_\cG}$ of $\cG$ transversely and in exactly one point;
 \item if $\cF$ has compact leaves, then $\gamma$ cuts the compact leaf containing $L$ transversely and in exactly $1$ point;
 \item $\gamma$ is a closed simple curve (even in the case $p=q$).
 \end{itemize}
 \begin{figure}[h]
\begin{center}
\def\svgwidth{0.9\columnwidth}
  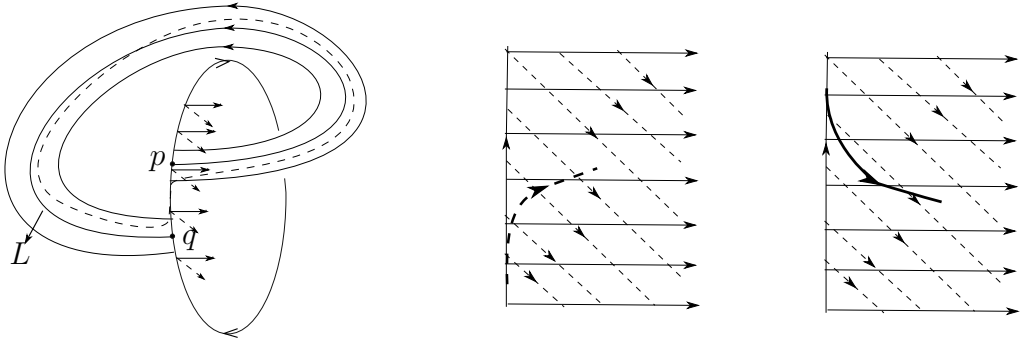
  \caption{In the first figure: the dash line is the transversal $\gamma$; the dash and real arrows on the circle pointing outside give the directions of $\calG$ and $\calF$ respectively. The second and the third figure show the good choice of curve and bad choice of curve respectively. }
\end{center}
\end{figure}
\begin{figure}[h]
\begin{center}
\def\svgwidth{0.35\columnwidth}
  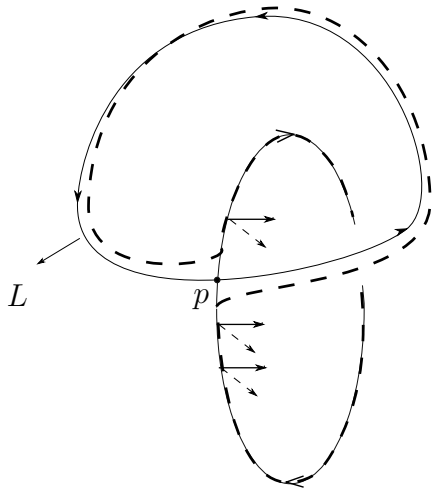
  \caption{The dash line is the transversal $\gamma$. The dash and real arrows on the circle pointing outside give the directions of $\calG$ and $\calF$ respectively.}
\end{center}
\end{figure}
\newpage
 Now $\gamma$ is a simple closed curve transverse to $\cG$ and has non-vanishing intersection number with a compact leaf of $\cG$,  and therefore $\gamma$ is not homotopic to the compact leaves of $\cG$.
 Lemma~\ref{compact leaf} implies therefore that $\gamma$ is a complete transversal of $\cG$.  The same argument show that, if $\cF$ has a compact leaf, then $\gamma$ is a complete transversal of $\cF$.
 Finally, if $\cF$ has no compact leaves, any closed transversal is a complete transversal, ending the proof.
\end{proof}

\section{Deformation of a foliation along its transverse foliation}

For any $C^1$ foliation $\calE$, we will denote by $\calE_x$  the leaf of $\calE$ through $x$. For any two points $x,y$ on a common
leaf of $\calE$, we denote $ d_{\calE}(x,y) $  as the distance between $x,y$ on the $\cE$-leaf.
  \begin{proposition}\label{push} Let $S=\R\times [0,1]$ be a horizontal strip on $\R^2$. Assume that $\tilde\calE$, $\tilde\calF$  and  $\tilde\calG$ are
    $C^1$ foliations on $S$ satisfying that
  \begin{itemize}
  \item the foliation $\tilde\cG$ is transverse to $\tilde \cF$ and $\tilde\cE$, that is,  $\tilde\calE\pitchfork\tilde\calG$ and $\tilde\calF\pitchfork\tilde\calG$;
   \item the foliations $\tilde\calE$, $\tilde\calF$ and $\tilde\calG$ are invariant under the map $(x,y)\mapsto(x+1,y)$;
  \item the foliations $\tilde\calE$ and $\tilde\calF$ have the same holonomy map from $\R\times\{0\}$ to $\R\times\{1\}$;
 \item Each leaf of each foliation intersects the two boundary components of $S$  transversely.
  \end{itemize}
   Then there exists a continuous family  $\{\Phi_t\}_{t\in[ 0 , 1 ]}$ of $C^1$ diffeomorphisms  on $\R\times[0,1]$ such that
   \begin{itemize} \item $\Phi_0=Id$;
   \item $\Phi_1(\tilde\calE)=\tilde\calF$;
   \item $\Phi_t(\tilde\calE)\pitchfork\tilde\calG$, for every $t\in[ 0 , 1 ]$;
   \item $\Phi_t$ commutes with the map  $(x,y)\mapsto(x+1,y)$, for any $t\in[0,1]$;
   \item $\Phi_t$ coincides with the identity map on $\RR\times\{0,1\}$, for any $t\in[0,1]$.
   \end{itemize}

 If furthermore $\tilde\cE$ and $\tilde\cF$ coincide in a neighborhood of the boundary $\RR\times\{0,1\}$ of $S$ then we can choose the family
 $\{\Phi_t\}_{t\in[0,1]}$ of diffeomorphisms so that there is a neighborhood of
 $\RR\times\{0,1\}$ on which the $\Phi_t$ coincides with the identity map, for any  $t\in[0,1]$.
   \end{proposition}
\begin{proof} By assumption, for each $x\in\RR\times\{0\}$, the leaf $\tilde\calE_x$  and the leaf $\tilde\calF_x$  have the same boundary.

\begin{claim}  For each $y\in\tilde\calE_x$, the leaf $\tilde\calG_y$  intersects $\tilde\calF_x$ at a unique point.
\end{claim}
\proof Since $\tilde\calE$ and $\tilde\calF$ are transverse to $\tilde\calG$, one can prove that every leaf of $\tilde\calG$ intersects every leaf of $\tilde\calE$
and $\tilde\calF$ in at most one point. If $y$ is an end point of $\tilde\cE_x$ it is also an end point of $\tilde \cF_x$, concluding.  Consider now $y$ in the interior of $\tilde\cE_x$.

Recall that $\tilde\calG_y$ is a segment joining  the two boundary components of $S$.  Thus  $S\setminus \tilde\calG_y$ has two connected
 components. Moreover each connected component of $S\setminus \tilde\cG_y$   contains exactly one end point of $\tilde\calE_x$.
 As  $\tilde\calF_x$ has the same end points as $\tilde\cE_x$, it intersects
  $\tilde\calG_y$.
\endproof

Now, we can define a map $h_{\tilde\calE}$ from $S$ to itself. For each $x\in S$, there exists a unique leaf of $\tilde\calF$ which has  the
same boundary as  $\tilde\calE_x$, by the claim above, $\tilde\calG_x$ intersects  that unique leaf of $\tilde\calF$ at only one point and we
 denote it as $h_{\tilde\calE}(x)$ (see Figure 5 below).
 \begin{figure}[h]
\begin{center}
\def\svgwidth{0.7\columnwidth}
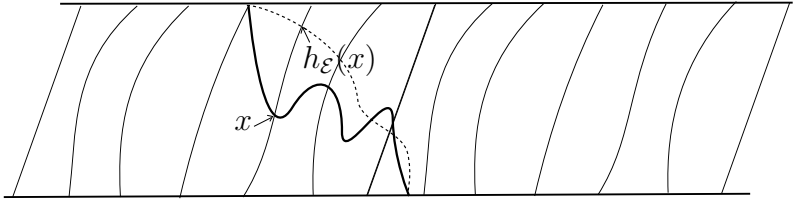
  \caption{The light line, the dark line and dash line  denote the leaves of $\tilde\calG$, $\tilde\calE$ and $\tilde\calF$ respectively.}
\end{center}
\end{figure}

 Since $\tilde\calE$, $\tilde\calF$ and $\tilde\calG$ are $C^1$-foliations, $h_{\tilde\calE}$ is a $C^1$ map, and its inverse $h_{\tilde \cF}$ is obtained by reversing the roles of $\tilde\cE$ and $\tilde \cF$,
 proving that $h_{\tilde\cE}$ is a diffeomorphism.
   Since each foliation is invariant under horizontal  translation $(r,s)\mapsto(r+1,s)$, the diffeomorphisms
     $h_{\tilde\calE}$ and $h_{\tilde\calF}$ commute with the map  $(r,s)\mapsto(r+1,s)$.

  Since $x$ and $h_{\tilde\calE}(x)$ are on the same $\tilde\calG$ leaf, the map  $d_{_\cG}(x,h_{\tilde\calE}(x))$ is well defined from $S$ to $\R$ and
  one can check that it is a $C^1$ map
  which is invariant under the translation $(r,s)\mapsto(r+1,s)$.

  Now, for each $t\in[0,1]$, we define $ \Phi_t(x) $ as the point, in the segment  joining $x$ to $h_{\tilde\calE}(x)$ in  the leaf $\tilde\calG_x$,  so that
  $$d_{\tilde\calG}(x,\Phi_t(x))=t \cdot d_{_{\tilde\cG}}(x,h_{\tilde\calE}(x)) \mbox{ and }d_{\tilde\calG}(\Phi_t(x),h_{\tilde\calE}(x))=(1-t)\cdot d_{_{\tilde\cG}}(x,h_{\tilde\calE}(x)).$$

Then, we have that $\Phi_0=Id$, $\Phi_1=h_{\tilde\calE}$ and $\Phi_t$  commutes with the horizontal translation $(r,s)\mapsto(r+1,s)$ and preserves each leaf of the foliation $\tilde\cG$.
One easily check  that $\Phi_t$ is of class $C^1$  and depends continuously on $t$. Furthermore,  its derivative along the leaves of $\tilde\cG$ does not vanish,
so that $\Phi_t$ is a diffeomorphism restricted to every leaf of $\tilde\cG$.
As $\Phi_t$ preserves every leaf of $\tilde\cG$, one deduces that $\Phi_t(\tilde\cE)$ is transverse to $\tilde\cG$ and $\Phi_t$ is a diffeomorphism of $S$.
Thus, $\{\Phi_t\}_{t\in[0,1]}$ is the announced continuous path of $C^1$ diffeomorphisms of $S$.

\end{proof}

For any  $C^1$ simple closed curve $\gamma$ on $\T^2$ whose homotopy class is non-trivial, we can cut the torus along $\gamma$
to get a cylinder. The universal cover of the cylinder is a strip denoted by $S_\gamma$ and diffeomorphic to $\R\times[0,1]$.
For any $C^1$ foliation $\calE$ on $\T^2$  transverse to $\gamma$, one denotes by $\tilde\cE$
 the lift of $\cE$ on $S_\gamma$.

\begin{corollary}\label{c.push} Let $\calE$, $\calF$ and $\calG$ be three $C^1$ foliations on $\T^2$. Assume that:
\begin{itemize} \item $\calE\pitchfork\calG$ and $\calF\pitchfork\calG$;
\item there exists a $C^1$ simple closed curve $\gamma$ such that
\begin{enumerate}\item the curve $\gamma$ is a complete transversal of the foliations  $\calE$, $\calF$ and $\calG$;
 \item the lifted foliations $\tilde{\calE}$ and $\tilde{\calF}$ have the same holonomy map  defined from one boundary component of $S_{\gamma}$ to the
 other;
  \end{enumerate}
  \end{itemize}
   Then there exists a continuous family of $C^1$ diffeomorphisms $\{\Phi_t\}_{t\in[0\, , \,1]}\subset\diff^{1}(\T^2)$ such that
   \begin{itemize} \item[--] $\Phi_0=Id$;
   \item [--]$\Phi_t(\calE)\pitchfork\calG$, for every $t\in[ 0 , 1 ]$;
   \item [--]$\Phi_1(\calE)=\calF$.
   \end{itemize}
\end{corollary}
\begin{proof}[Sketch of proof]If we just apply Proposition~\ref{push}, one obtains a family of  homeomorphisms of $\TT^2$ which are $C^1$ diffeomorphisms on the complement of
$\gamma$ and which coincide with the identity map on $\gamma$ and satisfy all the announced properties.  Thus the unique difficulty is the regularity along $\gamma$.
For that we check that the construction in the proof of Proposition~\ref{push} can be done on the whole universal cover of $\TT^2$ commuting with all the deck transformations,
leading to diffeomorphisms on $\TT^2$.
\end{proof}

\section{Deformation process for  transverse foliations without  parallel compact leaves:  proof of Theorem \ref{thm.nonparallel}}

\subsection{Separating transverse foliations by two linear ones, and proof of Theorem~\ref{thm.nonparallel}}

\begin{theorem}\label{t.separating} Let $\cF$ and $\cG$ be two transverse $C^1$ foliations on   $\TT^2$ without parallel compact leaves.
Then there are two affine foliations $\cH$ and $\cI$ on $\TT^2$ and a diffeomorphism $\theta\colon \TT^2\to\TT^2$ so that
\begin{itemize}
 \item the foliations $\theta(\cF)$, $\theta(\cG)$, $\cH$ and $\cI$ are pairwise transverse;
 \item there are local orientations  of the foliations at any point $p\in\TT^2$ so that
 \begin{itemize}
 \item  $\theta(\cF)$ and $\theta(\cG)$ cut $\cH$ with the same orientation;
 \item $\theta(\cF)$ and $\theta(\cG)$ cut $\cI$ with opposite orientations.
 \end{itemize}
\end{itemize}
\end{theorem}
The two affine foliations $\cH$ and $\cI$ divide the tangent space $T_pM$ at each point $p\in \TT^2$ into four quadrants, and  Theorem~\ref{t.separating}
asserts that the tangent lines at $p$ of  $\theta(\cF)$ and $\theta(\cG)$ are contained in different quadrants.

The proof of Theorem~\ref{t.separating} is the aim of the whole section. Let us first deduce the proof of Theorem~\ref{thm.nonparallel}
\begin{proof}[Proof of Theorem~\ref{thm.nonparallel}]
We consider  two transverse $C^1$ foliations $\cF$ and $\cG$ on $\TT^2$ without compact leaves and the diffeomorphism $\theta$ and the affine foliations $\cH$ and $\cI$ given by Theorem~\ref{t.separating}.
Consider any vector $u\in\RR^2$ and let $T_u$ be the affine translation of $\TT^2$ directed by $u$, that is $T_u(p)=p+u$.

\begin{claim} For any $u\in\RR^2$,  the foliation $T_u(\theta(\cF))$ is transverse to $\theta(\cG)$

\end{claim}
\begin{proof} The foliations $\cH$ and $\cI$ are invariant by $T_u$, and the quadrants defined by $\cH$ and $\cI$ are preserved by $T_u$   so that $T_u(\theta(\cF))$ is still transverse to
both $\cH$ and $\cI$ and its tangent bundle is contained in the same quadrants as $\theta(\cF)$, and therefore  $T_u(\theta(\cF))$ is not contained in the same quadrants as the tangent bundle of $\theta(\cG)$.

Thus $T_u(\theta(\cF))$ is transverse to $\theta(\cG)$, concluding.
\end{proof}

Consider $(m,n)\in\ZZ^2=H_1(\TT^2,\ZZ)$ and let $u=(r,s)$ be the image of $(m,n)$ by the natural action of $\theta$ on $H_1(\T^2,\ZZ)$. The the announced loop of diffeomorphism
is  $\{\varphi_t=\theta^{-1}T_{tu}\theta\}_{t\in[0,1]}$.  Then $\varphi_t(\cF)$ is transverse to $\cG$ for every $t\in [0,1]$, $\varphi_0=\varphi_1= id_{_{\TT^2}}$, and the
loop $t\mapsto \varphi_t(p)$ is in the homology class of $(m,n)$ for every $p\in \TT^2$.
\end{proof}

Therefore, it remains  to prove Theorem~\ref{t.separating}.  The proof is divided into two main steps corresponding to the next subsections.

\subsection{Separating transverse foliations  by a circle bundle}

In this section,  consider two transverse foliations $\cF$ and $\cG$ without parallel compact leaves. We first choose a coordinate
to make $\cG$ in a ``good position", then we  apply Proposition~\ref{push} to deform $\cF$ in ``good position", keeping $\cG$ invariant.

By Lemma \ref{st}, there exists a smooth simple closed curve $\gamma$ which is a complete transversal of $\calF$ and $\calG$.
The aim of this section is to prove next result which can be seen as the first step for proving Theorem~\ref{t.separating}.

\begin{theorem}\label{t.circles}  Let $\cF$ and $\cG$ be two transverse $C^1$ foliations on  $\TT^2$ and assume that they share the same complete transversal $\gamma$.  Then
there exists $\theta\in\diff^1(\T^2)$ such that
\begin{itemize}
\item  $\theta(\gamma)=\SS^1\times \{0\}$;
\item   Both $\theta(\cF)$ and $\theta(\cG)$ are  transverse to the
horizontal circle $\SS^1\times \{t\}$, for any $t\in\SS^1$.
\end{itemize}
\end{theorem}

Up to now,  $\cF$ and $\cG$ are two transverse foliations on $\TT^2$ which share the same complete transversal $\gamma$.
In particular,  $\cG$ is conjugated to the suspension of its holonomy (first return map) on $\gamma$. In other words, we can choose an
appropriate  coordinate on $\T^2=\mathbb{S}^1\times\mathbb{S}^1$ such that:
\begin{itemize}
\item  the circle $\gamma=\mathbb{S}^1\times\{0\}$ is a complete transversal for $\cF$ and $\cG$;
 \item the foliation $\calG$ is everywhere transverse to the horizontal circles;
 \item the foliation $\cG$ is vertical in a small neighborhood of $\mathbb{S}^1\times\{0\}$.
\end{itemize}

 Under this coordinate, we cut the torus along $\gamma$ and we get a cylinder $\mathbb{S}^1\times[0,1]$. Thus $\TT^2$ is obtained
 from $\SS^1\times[0,1]$ by identifying $(x,0)$ with $(x,1)$, for $x\in \SS^1$.

 Now, we take a universal cover of that cylinder,
 we get a strip $S=\R\times [0,1]$. The foliations $\calF$ and $\calG$ can be lifted as two foliations $\tilde{\calF}$ and $\tilde{\calG}$ on $S$,
  respectively. Moreover, $\tilde{\calG}$ is everywhere transverse to the horizontal direction.

 The proof of Theorem~\ref{t.circles} has two steps: first we build a foliation $\cE$ on $\T^2$   transverse to $\cG$ and to the horizontal foliation,
 so that $\cE$ has the same holonomy as $\cF$.
 Then we push $\cF$ on $\cE$ by a diffeomorphism preserving $\cG$, by using Corollary~\ref{push}.  Thus the main step of the proof is:

\begin{proposition}\label{p.transverse to horizontal}With the notations above,  there exist $\epsilon>0$ and  a $C^1$ foliation $\calE$ transverse to $\calG$ on
 $\T^2$  such that the foliation $\tilde{\calE}$ induced by $\calE$ on the strip $S=\RR\times[0,1]$ satisfies:
\begin{itemize}
 \item the foliation $\tilde{\calE}$ is transverse to the horizontal direction;
\item For any $x\in\R\times\{0\}$, we have that
$$\tilde{\calE}_{x}\cap (\R\times([0,\epsilon]\cup[1-\epsilon,1]))=\tilde{\calF}_{x}\cap (\R\times([0,\epsilon]\cup[1-\epsilon,1])).$$
\end{itemize}
\end{proposition}
\begin{Remark} The last item of Proposition \ref{p.transverse to horizontal} means  that
\begin{itemize} \item the foliations $\cE$ and $\cF$ coincide in a neighborhood of $\gamma$
\item the holonomy maps from $\R\times\{0\}$
 to $\R\times\{1\}$ associated to $\tilde{\calE}$ and $\tilde{\calF}$ are the same.
 \end{itemize}
\end{Remark}
\begin{proof}
 We denote by  $$f, g:\R\times\{0\}\mapsto\R\times\{1\}$$ the $C^{1}$ holonomy maps of  $\tilde{\calF}$ and $\tilde{\calG}$  respectively.

As $\tilde\cF$ is transverse to $\tilde \cG$, we have that  $f(x) \neq  g(x), \mbox{ for any } x\in\R\times\{0\}$. Hence, we can assume that $f(x)>g(x)$ for any
$x\in\R\times\{0\}$ (the other case is similar).

\vskip 2mm

We denote by $g_t\colon \RR\to \RR$ the holonomy of $\tilde \cG$ from $\RR\times\{0\}$ to $\RR\times \{t\}$. In particular
$g_0$ is the identity map and $g_1=g$. Our assumption that $\cG$ is vertical close to the boundary, implies that $g_t$ is the identity map
for $t$ small enough and $g_t=g$ for $t$ close to $1$.

Let $\psi_0\colon S\to S$ be defined by $(x,t)\mapsto (g_t(x),t)$.
Then $\psi_0$ is a diffeomorphism which commutes with the translation $T_1\colon (x,t)\mapsto (x+1,t)$.

Consider the foliations $\tilde \cG_0= \psi_0^{-1}(\tilde G)$ and $\tilde\cF_0=\psi_0^{-1}(\tilde \cF)$.
Now we have
\begin{itemize}
\item $\tilde\cG_0$ is the vertical foliation;
\item $\tilde\cF_0$ is a $C^1$ foliation transverse to the vertical foliation and transverse to the boundary of $S$,
and invariant by the translation $T_1$.  We denote by $\cF_0$ its quotient on the annulus $\SS^1\times [0,1]$.
\item every leaf of $\tilde \cF_0$ goes from $\RR\times \{0\}$ to $\RR\times \{1\}$ so that the holonomy map $f^0$ is well
defined and $f^0=g^{-1}\circ f$.  Our assumption $f(x)>g(x)$ means that $f^0(x)>x$ for every $x$.
\end{itemize}

As $\cF_0$ is transverse to the boundary of $\SS^1\times [0,1]$, there is $\delta>0$ so that $\tilde \cF_0$ is transverse to the horizontal foliation
on $\RR\times [0,\delta]$ and on $\RR\times [1-\delta,1]$.  Thus the holonomy  $f^0_t\colon \RR\times\{0\}\to \RR\times\{t\}$ of the foliation $\tilde \cF_0$
is well defined for $t\in[0,\delta]\cup[1-\delta,1]$ and satisfies:
\begin{itemize}
 \item $f^0_t(x)>x$ for $t>0$, and moreover $f^0_{t_1}(x)<f^0_{t_2}(x)$ for $t_1,t_1\in[0,\delta]\cup[1-\delta,1]$ and $t_1<t_2$.
 \item The map $(x,t)\mapsto f^0_t(x)$ is $C^1$ and $\frac{\partial f^0_t(x)}{\partial t} >0$ (because $\tilde\cF_0$ is transverse to the vertical foliation).
\end{itemize}

Consider $\varepsilon>0$ so that
\begin{itemize}
\item $$\varepsilon < \inf_{x\in\RR}\{ f^0_{1-\delta}(x)- f^0_\delta(x), \mbox{ for } x\in\RR \}$$
 \item $$\varepsilon < \inf\left\{ \frac{\partial f^0_t(x)}{\partial t}, \mbox{ for } x\in\RR \mbox{ and }t\in[0,\delta]\cup[1-\delta,1]\right\}$$
 \end{itemize}

With this choice of $\varepsilon$,  one can easily check the following inequalities
\begin{claim}

\begin{itemize}
 \item for any $t\in[0,\delta]$ and $x\in \RR$,  one has
 \begin{equation}\label{e.1}
 f^0_t(x)<\frac{f^0_{1-\delta}(x)+ f^0_\delta(x)}{2} + (t-\frac12)\varepsilon;
 \end{equation}
 \item for any $t\in[1-\delta,1]$ and $x\in \RR$,  one has
 \begin{equation}\label{e.2}
 \frac{f^0_{1-\delta}(x)+ f^0_\delta(x)}{2} + (t-\frac12)\varepsilon< f^0_t(x).
 \end{equation}
\end{itemize}

\end{claim}

Let $\alpha\colon [0,1]\to [0,1]$ be a smooth function so that:
\begin{itemize}
 \item $\alpha(t) \equiv 1$, for $t$ close to $0$ and close to $1$;
 \item $\alpha(t)\equiv 0$,  for $t\in[\frac{\delta}2, 1-\frac{\delta}2]$;
 \item $\frac{d\alpha}{dt}\leq 0$ on $[0,\delta]$ and $\frac{d\alpha}{dt}\geq 0$ on $[1-\delta,1]$.
\end{itemize}

For $x\in\RR$ and $t\in[0,1]$,  we define $h_t(x)$ as follows
\begin{itemize}
\item If $t\in [0,\delta]\cup [1-\delta,1]$ then
$$h_t(x)= \alpha(t) f^0_t(x) + (1-\alpha(t))\left( \frac{f^0_\delta(x)+f^0_{1-\delta}(x)}{2} + \varepsilon(t-\frac 12)\right);$$
 \item if $t\in[\delta,1-\delta]$ then
 $$h_t(x)= \frac{f^0_\delta(x)+f^0_{1-\delta}(x)}{2} + \varepsilon(t-\frac 12).$$
\end{itemize}

\begin{claim}The map $\psi_1\colon (x,t)\mapsto (h_t(x),t)$ is well defined  and is a $C^1$ diffeomorphism of $S$ such that
\begin{itemize}\item $\psi_1$ preserves the horizontal
foliation and commutes with the translation $T_1$;
\item
$$\frac{\partial h_t(x) }{\partial t}>0, \mbox{ for every } (x,t)\in S.$$
\end{itemize}
\end{claim}

\begin{proof}One easily checks that $\psi_1$ is continuous  and of class $C^1$.  The formula gives also that $\psi_1$ commutes with $T_1$.
Now
$$\frac{\partial h_t(x) }{\partial x}=\frac12( \frac{\partial f^0_\delta(x)}{\partial x}+\frac{\partial f^0_{1-\delta}(x) }{\partial x})>0 \mbox{ if }t\in[\delta,1-\delta]$$
and

$$\frac{\partial h_t(x) }{\partial x}= \alpha(t)\frac{\partial f^0_t(x)}{\partial x}    +
(1 -\alpha(t))\cdot \frac12\left( \frac{\partial f^0_\delta(x)}{\partial x}+\frac{\partial f^0_{1-\delta}(x) }{\partial x}\right) >0 \mbox{ if }t\notin[\delta,1-\delta].$$

This shows that $h_t$ is a diffeomorphism of $\RR$ for every $t\in[0,1]$.

It remains to prove the last item of the claim.

$$\frac{\partial h_t(x) }{\partial t}=\varepsilon>0, \mbox{ if }t\in[\delta,1-\delta]$$

 and  if  $t\notin[\delta,1-\delta]$ the derivative $\frac{\partial h_t(x) }{\partial t}$ is equal to :
 $$\frac{\ud\alpha(t)}{\ud t} \left(  f^0_t(x)- \left(\frac12  (f^0_\delta(x)+f^0_{1-\delta}(x)+\varepsilon(t-\frac 12)\right)\right)+ \alpha(t)\frac{\partial f^0_t}{\partial t} + (1-\alpha(t))\varepsilon.$$

The   last  two terms of this sum are positive, as product of positive numbers.

For $t\in[0,\delta]$ the first term is  product of two negative numbers, as the derivative of $\alpha$ is negative and (\ref{e.1}) implies:
$$ f^0_t(x)- \left(\frac12  (f^0_\delta(x)+f^0_{1-\delta}(x)+\varepsilon(t-\frac 12)\right)<0.$$

For $t\in [1-\delta,1]$ the first term is  product of two positive numbers,
as the derivative of $\alpha$ is positive and (\ref{e.2}) implies:
$$ f^0_t(x)- \left(\frac12  (f^0_\delta(x)+f^0_{1-\delta}(x)+\varepsilon(t-\frac 12)\right) >0.$$

Thus $\frac{\partial h_t(x) }{\partial t}>0$ for every $(x,t)$.
\end{proof}

Now the foliation $\tilde \cH$ defined as the image of the vertical foliation  by $\psi_1$  satisfies:
\begin{itemize}
 \item $\tilde \cH$ is transverse to the horizontal foliation.
 \item $\tilde \cH$ is transverse to the vertical foliation (that is, to $\tilde\cG_0$).
 \item its holonomy from $\RR\times \{0\}$ to $\RR\times \{t\}$ is $h_t$.  In particular, it coincides with $f_t$ for $t$ so that $\alpha(t)=1$, that
 is, in the neighborhood of $\RR\times \{0\}$ and $\RR\times\{1\}$.
 \item as a consequence of the previous item, the foliation $\tilde \cH$ coincides with $\tilde \cF_0$ in the neighborhood of $\RR\times \{0\}$ and $\RR\times\{1\}$.
\end{itemize}

We can now finish the proof of Proposition~\ref{p.transverse to horizontal}: the announced foliation on the strip $S$ is $\tilde \cE=\psi_0(\tilde\cH)$. This foliation
is invariant under the translation $T_1$, so it passes to the quotient in a foliation $\cE$ on the annulus $\SS^1\times [0,1]$.  As $\tilde \cE$ coincides with $\tilde \cF$ in a neighborhood of the boundary of $S$,
one gets that $\cE$ coincides with $\cF$ on the boundary of the annulus, and therefore this foliation induces a $C^1$ foliation, still denoted by $\cE$ on the torus $\TT^2$.
\end{proof}

Next remark ends the proof of Theorem~\ref{t.circles}:
\begin{Remark}

According to  Proposition \ref{push}, there is a continuous path of diffeomorphisms $\{\varphi_t\}_{ t\in[0,1]}$  of $S$ so that:
\begin{itemize}
\item $\varphi_t$ commutes with the translation $T_1\colon (x,t)\mapsto (x+1,t)$;
\item $\varphi_0$ is the identity map;
\item for every $t$, $\varphi_t$ coincides with the identity map in a neighborhood of the boundary of $S$;
\item $\varphi_t(\tilde \cG)=\tilde \cG$ for every $t$;  in particular $\varphi_t(\tilde\cF)$ is transverse to $\tilde \cG$ for every $t$;
\item $\varphi_1(\tilde \cF)=\tilde\cE$.
\end{itemize}
\end{Remark}

We now state a small variation of the statement of Theorem~\ref{t.circles} which  follows (exactly as Theorem~\ref{t.circles})
from  of Propositions~\ref{p.transverse to horizontal} and~\ref{push}, and that we will use in a next section.

\begin{lemma}\label{l.circles}
Let $\cF$ and $\cG$ be two transverse $C^1$-foliations on the annulus $\SS^1\times [0,1]$.  Assume that
\begin{itemize}
 \item $\cG$ is transverse to every circle $\SS^1\times\{t\}$;
 \item $\cF$ is transverse to the boundary $\SS^1\times\{0,1\}$ and has no compact leaf in $\SS^1\times (0,1)$.
\end{itemize}

Then there is a $C^1$ diffeomorphism $\theta$ of $\SS^1\times [0,1]$ which coincides with the identity map in a neighborhood of the boundary
 and which preserves every leaf of $\cG$,  so that
$\theta(\cF)$ is transverse to every circle $\SS^1\times\{t\}$.

\end{lemma}

\subsection{Building the second linear foliation: end of the proof of Theorem~\ref{t.separating}}

\begin{proposition}\label{irrational foliation} Let $\cF$ and $\cG$ be two transverse $C^1$ foliations on   $\TT^2$ without  parallel compact leaves.
Assume that both $\calF$ and $\calG$ are transverse to the horizontal foliation. We endow $\cF$ and $\cG$ with orientations so that they cut the horizontal foliation
with the same orientation.

Then there exists a smooth ($C^\infty$) foliation $\calE$ on $\T^2$ such that
 \begin{itemize}\item the circle $\S^1\times \{0\}$  is a complete transversal to $\calE$;
  \item the holonomy map induced by $\calE$ on $\gamma$ has a diophantine rotation number;
  \item the foliation  $\calE$  is transverse to the foliations  $\calF$, $\calG$ and  to the horizontal direction.  We endow it with an orientation so that it
  cuts the horizontal foliation with the same orientation as $\cF$ and $\cG$.
   \item the foliation $\cE$ cuts $\cF$ and $\cG$ with opposite orientations.
   \end{itemize}
\end{proposition}

\begin{proof}

As already done before, we cut the torus along $\SS^1\times\{0\}$, getting an annulus, and we denote by $\tilde\cF$ and
$\tilde\cG$ the lift of $\cF$ and $\cG$ on the strip $\RR\times[0,1]$  which is the  universal cover of
the annulus.
We denote by $f$ and $g$ the holonomy maps from $\RR\times\{0\}$ to $\RR\times\{1\}$ associated to the lifted foliations $\tilde{\calF}$ and $\tilde{\calG}$, respectively.
By transversality of $\cF$ with $\cG$, we have that either $f(x)<g(x)$ for any $x$, or $f(x)>g(x)$ for any $x$.
Without loss of generality, we assume that $f(x)<g(x)$.
Let $\tau(f)$ and $\tau(g)$ be the translation numbers of $f$ and $g$.
\begin{claim} $\tau(f)\neq\tau(g)$.
\end{claim}
\begin{proof} We prove it by contradiction. Assume that $\tau(f)=\tau(g)$, then  $\tau(f)=\tau(g)$ is either  rational or irrational.
When they are irrational, since $f(x)<g(x)$, by Proposition \ref{p.comparing translation number}, we have that $\tau(f)<\tau(g)$,
 a contradiction.
 When they are both rational, then there exist $m,n\in\N$ such that $\tau(f)=\tau(g)=\frac{n}{m}$. Hence, there exist two
  points $x_{0},y_0\in\R$, such that
 \begin{displaymath} f^{m}(x_0)=x_0+n \textrm{ and } g^{m}(y_0)=y_0+n,
 \end{displaymath}
 which implies that there exist compact leaves of $\calF$ and $\calG$ that are in the homotopy class of $(m,n)$, contradicting the
 non-parallel assumption.
\end{proof}

We  endow $\tilde{\calF}$ and $\tilde{\calG}$ with orientations such that they point inward the strip $S$ at $\R\times\{0\}$ and point
outward at $\R\times\{1\}$, and $\cF$ and $\cG$ are endowed with the corresponding orientations.
 Let  $X$ and $Y$ be the unit vector fields tangent to $\calF$ and $\calG$ respectively, pointing to the
 orientation of the corresponding foliation, and $\tilde X$ and $\tilde Y$ be their lifts on $S$.

 \begin{claim} There are smooth vector fields $U$ and $V$ on $\TT^2$ so that
 \begin{itemize}
  \item at each point $x\in \TT^2$,  the vertical  coordinates of $U(x)$ and $V(x)$ are strictly positive.
  In particular, $U$ and $V$ are transverse to the horizontal foliation.
  We denote by $\tilde U$ and $\tilde V$ the lifts of $U$ and $V$ on the strip $S$.
  \item let $h$ and $k$ be the holonomies of   $\tilde U$ and $\tilde V$, respectively, from $\RR\times \{0\}$ to $\RR\times\{1\}$.  These holonomies commute with the translation $T_1$, and
  let $\tau(h)$ and $\tau(k)$ denote their translation numbers.  Then
  $$\tau(f)\leq \tau(h)<\tau(k)\leq \tau(g).$$
 \end{itemize}
 \end{claim}
 \begin{proof}Just consider a small enough $\varepsilon>0$ and consider smooth vector fields $U$ and $V$ arbitrarily $C^0$ close to $X+\varepsilon Y$ and $\varepsilon X+Y$, respectively.
 \end{proof}

 Now the vector fields $U_t=(1-t)U+ tV$, $t\in[0,1]$,  are all transverse to both foliations  $\cF$, $\cG$ and to the horizontal foliation,  and they cut $\cF$ and $\cG$ with opposite orientations.
 We denote by $\tilde U_t$ the lift of $U_t$ on the strip $S$.
 Let $\tau_t$ denote the translation number of the holonomy of $\tilde U_t$ from $\RR\times \{0\}$ to $\RR\times \{1\}$. According to Proposition~\ref{p.continuity of translation number},
 the map $t\mapsto \tau_t$ is a continuous monotonous function joining
 $\tau(h)$ to $\tau(k)$. As $\tau(h)<\tau(k)$ there is $t\in(0,1)$ for which $\tau_t$ is an irrational diophantine number, ending the proof.
\end{proof}

We  end the proof of Theorem~\ref{t.separating} by noticing that Theorem \ref{He} implies

\begin{lemma} Let $\cE$ be a smooth foliation on $\TT^2$ transverse to the horizontal foliation and so
that its holonomy on $\SS^1\times \{0\}$ is a diffeomorphism with an irrational diophantine
rotation number.  Then there is a diffeomorphism $\theta$ of $\TT^2$ which preserves each horizontal circle $\SS^1\times \{t\}$, for any $t\in\SS^1$,  and satisfies  that
$\theta(\cE)$ is an affine foliation.
\end{lemma}

\section{Deformation process of parallel case and proof of Theorem \ref{thm.parallel}}\label{s.parallel}
We dedicate this whole section to give the proof of Theorem \ref{thm.parallel}. We sate a definition which is only used in this section.
\begin{definition}Given a $C^1$ foliation $\cE$ on the annulus $[0,1]\times\S^1=[0,1]\times \R/\Z$ without
compact leaves such that $\cE$  is transverse to the vertical circle $\{t\}\times\S^1$, for any $t\in[0,1]$.
The leaves of such a foliation $\cE$ are called
 \begin{itemize}
 \item[--] \emph{not increasing (resp. not  decreasing)}, if the lifted foliation $\tilde{\cE}$ on $[0,1]\times\R$ satisfies that every leaf of $\tilde{\cE}$ is not  increasing (resp. not decreasing);
 \item[--]\emph{non-degenerate increasing (resp. non-degenerate decreasing)}, if  the  lifted foliation $\tilde{\cE}$ on $[0,1]\times\R$ satisfies that
every leaf of $\tilde{\cE}$ is strictly increasing (resp. strictly decreasing) and transverse to the horizontal foliation $\{[0,1]\times\{t\}\}_{t\in\R}$.
\end{itemize}
\end{definition}
\subsection{Normal form for two transverse foliations with parallel compact leaves and proof of Theorem~\ref{thm.parallel}}

The aim of this section is the proof of Theorem~\ref{thm.parallel}.  The main step for this proof is the following result which puts any
pair of transverse $C^1$ foliations in a canonical position.

\begin{theorem}\label{t.normal-form}  Let $\cF$ and $\cG$ be two transverse $C^1$ foliations on $\TT^2$ admitting parallel compact leaves.
Then  there are an integer $k$,  a set of points $\{t_i\}_{i\in\ZZ/k\ZZ}$ in $\SS^1$ which are  cyclically ordered on $\SS^1$,  and  a diffeomorphism $\theta\colon \TT^2\to \TT^2$ so that
\begin{itemize}
 \item the foliations $\theta(\cF)$ and $\theta(\cG)$ are transverse to $\{t_i\}\times S^1$,  for any $i\in\ZZ/k\ZZ$;
 \item for each $i\in\ZZ/k\ZZ$, the restrictions of the foliations $\theta(\cF)$ and $\theta(\cG)$ to the annulus $C_i=[t_i,t_{i+1}]\times \SS^1$ satisfy one of the six possibilities below
 \begin{enumerate}
 \item $\theta(\cF)$ coincides with the horizontal foliation on $C_i$ and $\cG$ admits compact leaves in $C_i$;
 \item $\theta(\cG)$ coincides with the horizontal foliation on $C_i$ and $\cF$ admits compact leaves in $C_i$;
 \item the foliations $\theta(\cF)$ and $\theta(\cG)$ are transverse to the vertical foliation on $C_i$.  Furthermore, every leaf of $\theta(\cF)$ (resp. of $\theta(\cG)$) on $C_i$   is non-degenerate  increasing  (resp. not increasing);
 \item the foliations $\theta(\cF)$ and $\theta(\cG)$ are transverse to the vertical foliation on $C_i$.  Furthermore,   every leaf of $\theta(\cF)$ (resp. of $\theta(\cG)$) on   $C_i$
 is not increasing (resp. non-degenerate increasing);
 \item the foliations $\theta(\cF)$ and $\theta(\cG)$ are transverse to the vertical foliation on $C_i$.  Furthermore,  every leaf of $\theta(\cF)$ (resp. of $\theta(\cG)$)  on $C_i$
 is non-degenerate decreasing (resp. not decreasing);
 \item the foliations $\theta(\cF)$ and $\theta(\cG)$ are transverse to the vertical foliation on $C_i$.  Furthermore,  every leaf of $\theta(\cF)$ (resp. of $\theta(\cG)$) on
   $C_i$
 is not decreasing (resp. non-degenerate decreasing).
 \end{enumerate}
\end{itemize}
\end{theorem}

The proof of Theorem~\ref{t.normal-form} will be done in the next subsections.  We start below by ending the proof of Theorem~\ref{thm.parallel}.
\begin{proof}[Proof of Theorem~\ref{thm.parallel}] Let $\cF$ and $\cG$ be two $C^1$ foliations of $\TT^2$ admitting  parallel compact leaves,
and let $\alpha\in\pi_1(\TT^2)$ be the homotopy class of the compact leaves of $\cF$ and $\cG$.
Let $k>0$, $\{t_i\}_{i\in\ZZ/k\ZZ}$, and $\theta$ be the integer, the elements of $\SS^1$
 and the diffeomorphism given  by Theorem~\ref{t.normal-form}, respectively.

One easily checks  that there is at least one annulus of the type (1) or (2).  As a consequence, the compact leaves of $\theta(\cF)$ are isotopic to the vertical circle $\{0\}\times \SS^1$.

Consider any vertical vector $(0,t)$, for  $t\in\RR$,  and let $V_t$ be the vertical translation defined by $(r,s)\mapsto (r,t+s)$.  Then $V_t$ preserves each annulus $C_i$.  Now one can check, on each annulus $C_i$, that
$V_t(\theta(\cF))$ is transverse to $\theta(\cG)$.

Consider now $\beta\in \langle\alpha\rangle$, so that $\beta=n\alpha$ for some $n\in\ZZ$. Then the announced loop of diffeomorphisms is $\{\theta^{-1}\circ V_{nt}\circ\theta\}_{t\in[0,1]}$.
\end{proof}
\subsection{First decomposition in annuli}

By Theorem \ref{H}, the sets of compact leaves of $\calF$ and $\calG$ are all compact sets. We denote the unions of compact leaves
 of $\calF$ and $\calG$ as $K_F$ and $K_G$ respectively. Note that every compact leaf of $\cF$ is disjoint from any compact leaf of $\cG$,
 because they are in the same homotopy class, and by assumption, $\cF$ and $\cG$ are transverse.  Thus $K_G$ and $K_F$ are disjoint compact sets.

 The aim of this section is to prove Proposition~\ref{p.normal-form} below which is an important step for proving Theorem~\ref{t.normal-form}.

\begin{proposition}\label{p.normal-form} Let $\cF$ and $\cG$ be two transverse $C^1$-foliations on $\TT^2$ having parallel compact leaves.
Then there are $k_0$ and  a family $\{B_i\}_{i\in\ZZ/4k_0\ZZ}$ of annuli so that
\begin{itemize}
 \item each $B_i$ is an annulus diffeomorphic to $[0,1]\times\SS^1$ and embedded in $\TT^2$ whose boundary is transverse to both foliations $\cF$ and $\cG$.
 \item $B_i$ is disjoint from $B_{j}$ if $j\notin\{i-1,i,i+1\}$,  and $B_i\cap B_{i+1}$ consists in a common connected component of the boundaries $\partial B_i$ and $\partial B_{i+1}$.
 In particular,  the interiors of these $B_i$ are pairwise disjoint;
 \item each annulus $B_{2j+1}$ is disjoint from the compact leaves of $\cF$ and of $\cG$, that is
 $$B_{2j+1}\cap\left(K_F\cup K_G)\right)=\emptyset;$$
 \item each annulus $B_{4i}$ contains compact leaves of $\cF$ and is disjoint from the compact leaves of $\cG$;
 \item each annulus $B_{4i+2}$ contains compact leaves of $\cG$ and is disjoint from the compact leaves of $\cF$.
\end{itemize}
\end{proposition}

We say that a compact set $C$ is  \emph{a $\cF$-annulus} (resp. \emph{a $\cG$-annulus}) if we have the following:
 \begin{itemize}
 \item[--] the compact set $C$ is  diffeomorphic to either   $\SS^1$ or   $\SS^1\times [0,1]$;
 \item[--] the compact set $C$ is disjoint from $K_G$ (resp. of $K_F$);
 \item[--] the boundary of $C$ consists in compact leaves of $\cF$ (resp. of $\cG$).
\end{itemize}  We say that two compact leaves $L_1, L_2$ of $\cF$ (resp. of $\cG$) are
\emph{$K_G$-homotopic} (resp. \emph{$K_F$-homotopic}) if $L_1\cup L_2$ bounds a $\cF$-annulus (resp. \emph{a $\cG$-annulus}).

\begin{Remark}\label{r.F-homotopic}
 \begin{itemize}
  \item The union of two  non-disjoint $\cF$-annuli  is a $\cF$-annulus.
  \item two compact leaves of $\cF$ are $K_G$-homotopic if and only if  they are contained in the same $\cF$-annulus.
  \item there is $\delta>0$ so that any two   compact leaves of $\cF$ passing through points $x,y$ with $d(x,y)\leq \delta$ are $K_G$-homotopic.
 \end{itemize}
\end{Remark}

As a direct consequence of Remark~\ref{r.F-homotopic}, one gets
\begin{lemma}\label{l.F-annulus}
\begin{enumerate}
 \item The relation of  $K_G$-homotopy (resp. of $K_F$-homotopy) is an equivalence relation on $K_F$ (resp. of $K_G$).
 \item there are finitely many $K_G$-homotopy classes (resp. $K_F$-homotopy classes).
 \item There are $k\in \NN\setminus\{0\}$ and  pairwise disjoint compact sets $\{A_i\}_{i\in\ZZ/2k\ZZ}$ , so that
 \begin{itemize}
  \item $A_{2i}$ is an $\cF$-annulus and $A_{2i+1}$ is a $\cG$-annulus.
  \item For each $K_G$-homotopy class (resp. $K_F$-homotopy class)  of compact leaves of $\cF$ (resp. of $\cG$), there is a (unique) $i$ so that the class is precisely the set of compact leaves
  of $\cF$ (resp. of $\cG$) contained in $A_{2i}$ (resp. in $A_{2i+1}$).
  \item these $\{A_i\}$ are cyclically ordered in the following meaning: for any $i\in\ZZ/2 k\ZZ$,  the set  $\TT^2 \setminus (A_{i-1}\cup A_{i+1})$ consists precisely in two disjoint  open annuli such that one
  of them contains $A_i$ and is disjoint from $A_j$ for $j\neq i$.
 \end{itemize}
\end{enumerate}
\end{lemma}
\begin{proof} The set $\{A_i\}$ is  defined as the set of   the unions of  all the $\cF$-annuli containing  compact leaves of $\cF$ in a given $K_G$ homotopy class
 and
the unions of  all the $\cG$-annuli containing  compact leaves of $\cG$ in a given $K_F$ homotopy class.
Then $\{A_i\}$ bound a family of disjoint compact annuli (or circles) whose boundary are non-null homotopic  simple curves on $\TT^2$. Thus these curves are in the same homotopy class and
the annuli are cyclically  ordered on $\TT^2$. Thus, up to reorder the annuli, we assume that the order is compatible with the cyclic order.
Finally if $A_i$ is a  $\cF$ annulus  then $A_{i+1}$ cannot be a $\cF$-annulus, otherwise there would exist a $\cF$-annulus containing
both $A_i$ and $A_{i+1}$,
contradicting to the maximality of $A_i$.
\end{proof}

The annuli $A_{2i}$ and $A_{2i+1}$ will be called the \emph{maximal $\cF$-annuli} and \emph{maximal $\cG$-annuli}, respectively.

\begin{lemma}\label{l.Bi} With the hypotheses and terminology above, each maximal $\cF$-annulus (resp. maximal $\cG$-annulus) $A$  admits a base of neighborhoods $\{B_n\}_{n\in\NN}$ which are
diffeomorphic to $[0,1]\times \SS^1$ and whose boundaries are transverse to both $\cF$ and $\cG$.
\end{lemma}
\begin{proof} Assume for instance that $A$ is a maximal $\cF$-annulus. Its boundary consists in compact leaves of $\cF$, and in particular is transverse to $\cG$.
Furthermore any neighborhood $V$ of $A$ contains
an annulus $U$ which is a neighborhood of $A$ and satisfies that $U\setminus A$ is disjoint from $K_F$ and  $K_G$.
Now each connected component of $U\setminus A$ contains an embedded circle which consists in   exactly one segment of leaf of $\cF$ and one segment
of leaf of $\cG$.
Exactly as in Section~\ref{s.section}, we get a simple curve transverse to both $\cF$ and $\cG$ by smoothing such a curve.

One gets the announced annulus by considering such a transverse curve to both $\cF$ and $\cG$ in each connected component of $U\setminus A$.

\end{proof}
\begin{proof}[Proof of Proposition~\ref{p.normal-form}] The announced annuli $B_{4i}$ and $B_{4i+2}$ are pairwise disjoint neighborhoods of the
maximal $\cF$ annuli and maximal $\cG$-annuli, respectively,  given
by Lemma~\ref{l.Bi}. Each  annulus $B_{2j+1}$ is given by  the closure of  a connected component of $\TT^2\setminus\bigcup_i(B_{4i}\cup B_{4i+2})$.
\end{proof}

\subsection{In the neighborhoods of the maximal $\cF$ annuli}

The aim of this section is to prove the following Proposition which implies that, in the neighborhoods of the maximal $\cF$ and $\cG$-annuli,
one can put $\cF$ and $\cG$ in the position announced by Theorem~\ref{t.normal-form}.

\begin{proposition}\label{p.maximal} Let $\cF$ and $\cG$ be two transverse foliations on $\T^2$ having parallel compact leaves.
 Let $\{B_j\}_{\in\ZZ/4k\ZZ}$ be the annuli, which are     built in Proposition~\ref{p.normal-form} and whose boundaries are  transverse
to both $\cF$ and $\cG$,   and $A_{4i}$  (resp. $A_{4i+2}$) be  the maximal $\cF$-annuli (resp.$\cG$-annuli) contained in $B_{4i}$
 (resp. in $B_{4i+2}$), for $i\in  \ZZ/k\ZZ$.

Then there exists $\theta\in\diff^1(\TT^2)$ such that  for every $j\in\ZZ/4k\ZZ$,  one has $$\theta (B_j)=[\frac j{4k},\frac{j+1}{4k}]\times \SS^1;$$
 and for every $i\in\ZZ/k\ZZ$, one has:
\begin{itemize}
 \item the foliation $\theta(\cG)$ coincides with  the horizontal foliation  $\{[\frac{4i}{4k},\frac{4i+1}{4k}]\times\{t\}\}_{t\in\SS^1}$ on $\theta (B_{4i})=[\frac{4i}{4k},\frac{4i+1}{4k}]\times \SS^1 $;
 \item there are $\frac{4i}{4k} <a_{4i}\leq b_{4i}<\frac{4i+1}{4k}$ such that  $\theta(A_{4i})=[a_{4i},b_{4i}]\times \SS^1$. In particular $\{a_{4i}\}\times\SS^1$ and $\{b_{4i}\}\times \SS^1$ are compact leaves
 of $\theta(\cF)$;
 \item the foliation $\theta(\cF)$ is transverse to the vertical circle  $\{r\}\times \SS^1$, for any $r\in [\frac {4i}{4k}, a_{4i})\cup (b_{4i},\frac{4i+1}{4k}]$.
 \end{itemize}
 and similarly:
 \begin{itemize}
 \item the foliation $\theta(\cF)$ coincides with  the horizontal foliation  $\{[\frac{4i+2}{4k},\frac{4i+3}{4k}]\times\{t\}\}_{t\in\SS^1}$ on $\theta ( B_{4i+2})=[\frac{4i+2}{4k},\frac{4i+3}{4k}]\times \SS^1 $
 \item there are $\frac{4i+2}{4k} <a_{4i+2}\leq b_{4i+2}<\frac{4i+3}{4k}$ so that  $\theta(A_{4i+2})=[a_{4i+2},b_{4i+2}]\times \SS^1$. In particular $\{a_{4i+2}\}\times\SS^1$ and $\{b_{4i+2}\}\times \SS^1$ are compact leaves
 of $\theta(\cG)$;
 \item the foliation $\theta(\cG)$ is transverse to the vertical circle  $\{r\}\times \SS^1$, for any $r\in [\frac {4i+2}{4k}, a_{4i+2})\cup (b_{4i+2},\frac{4i+3}{4k}]$.
 \end{itemize}

\end{proposition}

Proposition~\ref{p.maximal} is a straightforward consequence of Lemma~\ref{l.maximal} below:

\begin{lemma}\label{l.maximal} Let $\cF$ and $\cG$ be two transverse $C^1$-foliations of the annulus $[0,1]\times \SS^1$  so that
the boundary $\{0,1\}\times \SS^1$ is transverse to both $\cF$ and $\cG$.  Assume  that $\cG$ has no compact leaves (in $(0,1)\times \SS^1$)  and    $\cF$ admits
compact leaves in  $(0,1)\times \SS^1$.

  Then there exists  $\theta\in\diff^1([0,1]\times \SS^1)$  so that
\begin{enumerate}
 \item the foliation $\theta(\cG)$ is the horizontal foliation $\{[0,1]\times\{t\}\}_{t\in\SS^1}$.
 \item there are $0<a\leq b<1$ so that $\{a\}\times\SS^1$ and $\{b\}\times \SS^1$ are compact leaves
 of $\theta(\cF)$ and  every compact leaf of $\theta(\cF)$ is contained in $[a,b]\times \SS^1$;
 \item the foliation $\theta(\cF)$ is transverse to the vertical circles $\{r\}\times \SS^1$ for $r\notin[a,b]$.
\end{enumerate}
\end{lemma}

The proof of Lemma~\ref{l.maximal} uses the Lemma~\ref{l.analysis} below
\begin{lemma}\label{l.analysis} For any  continuous function $\varphi:[0,1]\mapsto[0,+\infty)$ such that $\varphi>0$ on $(0,1)$, and any interval $(c, d)\subset(0,1)\subset \S^1$,
 there exists   $\theta\in\diff^{\infty}(\RR\times\S^1)$ such that
 \begin{itemize}
 \item the diffeomorphism $\theta$ coincides with the identity map out of $[0,1]\times \SS^1$;
 \item $\theta([0,1]\times\{y\})=[0,1]\times\{y\}$;
 \item $D\theta(\frac{\partial}{\partial y})=\frac{\partial}{\partial y}+a(x,y) \frac{\partial}{\partial x}$;
 \item $a(x,y)>0$, for any $(x,y)\in(0,1)\times (c,d)$;
 \item $a(x,y)>-\varphi(x)$, for any $(x,y)\in (0,1)\times \SS^1$.
 \end{itemize}
\end{lemma}
\begin{proof}
We fix $c<d$ and take a point $e\in(d,1)$.
We take $\theta(x,y)=(x+\alpha(x)\beta(y), y)$ where $\alpha\colon\RR\to[0,+\infty)$ and $\beta\colon\SS^1\to [0,1]$ are smooth functions so that
\begin{itemize}\item  $\alpha(x)$ is defined on $\R$ and equals to zero in $(-\infty,0]\cup[1,+\infty)$;
 \item $0<\alpha(x)<\varphi(x)$ in the set $(0,1)$ and $\alpha\equiv 0$ out of $[0,1]$ (the existence of such a function is not hard to check);
 \item the derivative $\alpha'(x)$ is everywhere strictly  larger than $-1$;
\item $\beta(y)$ is  equal to zero in the set $[0, c]\cup [e,1]$;
\item the derivative $\beta'(y)$ is strictly positive for  $y\in (c,d)$;
\item the derivative $\beta'(y)$ is larger than $-1$ everywhere.
\end{itemize}
With this choice, one gets that the restriction of $\theta$ to any horizontal line has a non-vanishing derivative, hence is a diffeomorphism.  One deduces that $\theta$ is a diffeomorphism
of $[0,1]\times \SS^1$.  Furthermore,
the function $a(x,y)$ in the statement is $\alpha(x)\cdot\beta'(y)$ which is strictly  positive on $(0,1)\times (c,d)$ 
and larger than $-\alpha(x)>-\varphi(x)$ for $x\in(0,1)$, concluding the proof.
\end{proof}
\begin{Remark}\label{r.analysis} In the proof of Lemma~\ref{l.analysis}, if we define $\theta_t$ by $$\theta_t(x,y)=(x+t\alpha(x)\beta(y), y), \textrm{ for any $t\in[0,1]$},$$
 one gets a continuous family of
diffeomorphisms for the $C^\infty$ topology so that $\theta_0$ is the identity map and every $\theta_t$, $t\neq 0$, satisfies the conclusion of Lemma~\ref{l.analysis}.  In particular, in Lemma~\ref{l.analysis}
one may choose $\theta$ arbitrarily $C^\infty$ close to identity.

\end{Remark}

\begin{proof}[Proof of Lemma~\ref{l.maximal}]

As $\cG$ is transverse to the boundary and has no compact leaves in $[0,1]\times \SS^1$, then as a simple corollary of Proposition~\ref{p.holonomies} one gets that, up to consider the images
of $\cF$ and $\cG$ by a diffeomorphism of the annulus, we may assume that $\cG$ is the horizontal foliation and that there are compact leaves $\{a\}\times \SS^1$ and $\{b\}\times \SS^1$, $0<a\leq b<1$, so that the
compact leaves of $\cF$ are contained in $[a,b]\times \SS^1$. In other words, we may assume that items (1) and (2) are already satisfied.  It remains to get item (3),
that is, to get the transversality of $\cF$
with the vertical fibers out of $[a,b]\times \SS^1$.

We first show
\begin{claim} There is a $C^1$ foliation $\cH$ defined in a neighborhood of the compact leaf $\{a\}\times \SS^1$ so that
\begin{itemize}
 \item the  leaves of $\cH$ are transverse to the horizontal foliation $\cG$;
 \item $\{a\}\times \SS^1$ is a compact leaf of $\cH$;
 \item the holonomies $h$ and $f$ 
 on the transversal $[0,1]\times \{0\}$ for the foliations $\cH$ and $\cF$ are equal;
 \item  the foliation $\cH$ is transverse to the vertical circle  $\{r\}\times \SS^1$, for $r<a$.
\end{itemize}
\end{claim}
\begin{proof}
We fix $0<\epsilon<1/2$ and a function $\alpha\colon[0,1] \to [0,1]$ so that $\alpha\equiv 0$ in  $[0,\varepsilon]$, $\alpha\equiv 1$ in $[1-\varepsilon,1]$ and $\alpha'(s)>0$ for
$s\in (\varepsilon,1-\varepsilon)$.

Consider the foliation $\cH_0$,  defined in a neighborhood of the compact leaf $\{a\}\times \SS^1$, whose holonomy map
 $h_s\colon [0,1]\times\{0\}\to [0,1]\times \{s\}$, for any $s\in\S^1$, is defined by
$r\mapsto \alpha(s) f(r)+ (1-\alpha(s)) r$,
where $f\colon [0,1]\times\{0\}\to[0,1]\times\{0\}$ is the holonomy map of $\cF$.

As $\cF$ has no compact leaves on $[0,a)\times \SS^1$, one gets that $f(r)\neq r$ for every $r<a$. Thus, by the choice of  $\cH_0$, we have that
\begin{itemize}
\item  $\cH_0$ is transverse to the horizontal foliation everywhere;
\item $\cH_0$ is  transverse to the vertical
foliation at each point $(r,s)$ with $r<a$ and $s\in (\varepsilon, 1-\varepsilon)$;
 \item $\cH_0$ is  vertical for $s$ in the interval $[0,\varepsilon]\cup [1-\varepsilon,1]= [-\varepsilon,\varepsilon]\subset \RR/\ZZ=\SS^1$.
\end{itemize}

We fix an interval $[e,f]\subset \SS^1$ disjoint from $[-\varepsilon,\varepsilon]$. The foliation $\cH_0$ is directed by a vector of the form
$\frac{\partial}{\partial s}+\delta(r,s)\frac{\partial}{\partial r}$,  where the function $\delta$ is continuous and non-vanishing on $[0,a)\times [e,f]$.
We define $\varphi(r)=\inf_{s\in[e,f]} |\delta(r,s)|$. By the absolute continuity of $\delta$, the map $\varphi$ is continuous and positive for $r<a$.
The map $\varphi$ is only defined on a small neighborhood of $a$,  and we extend it to $[0,a]$ as a continuous function which is positive on $(0,a)$.

Applying Lemma~\ref{l.analysis} to $\varphi$ and to  an interval $(c,d)$ containing $[-\varepsilon,\varepsilon]$ and disjoint from $[e,f]$, one gets a smooth
diffeomorphism $\theta_0$ of $[0,a]\times \SS^1$,
preserving each horizontal leaf,  such that $\theta_0(\cH_0)$ is transverse to the vertical foliation on $[0,a)\times\SS^1$, concluding the proof of the claim.

\end{proof}

The foliation $\cH$ defined by the claim in a neighborhood of $\{a\}\times \SS^1$, is conjugated to $\cF$ by a diffeomorphism preserving the compact leaf
$\{a\}\times \SS^1$ and every horizontal segment $[0,1]\times \{s\}$.
We can do the same in a neighborhood of the compact leaf $\{b\}\times \SS^1$.

Thus there is a diffeomorphism $\theta_1$  of $[0,1]\times \SS^1$ preserving the leaves $\{a\}\times \SS^1$ and $\{b\}\times \SS^1$ and preserving every horizontal
segment $[0,1]\times \{s\}$, and there is $\varepsilon >0$
so that $\theta_1(\cF)$
is transverse to the  vertical circles on $[a-\varepsilon, a)\times \SS^1$ and on $(b,b+\varepsilon]\times \SS^1$.

For concluding the proof,  it remains to put the foliation $\cF$ transverse to the vertical circles on the annuli $[0, a-\varepsilon]\times\S^1$ and $[b+\varepsilon, 1]\times\S^1$.
On each of these annuli, we have that
\begin{itemize}\item $\cG$ is a foliation transverse to the circle bundle;
\item $\cF$ is transverse everywhere to $\cG$;
\item Both $\cF$ and $\cG$ are transverse to the boundary and have no compact leaf.
\end{itemize}
Thus applying Lemma~\ref{l.circles} to these annuli,  one gets diffeomorphisms which preserve each leaf of $\cG$ and equal to the identity map on the boundary,
such that these diffeomorphisms send $\cF$ on a foliation
transverse to the circle bundle,  concluding the proof.
\end{proof}

Let us add a statement that we will not use,  but it is obtained by a slight modification of the proof of Lemma~\ref{l.maximal}:
\begin{corollary} Let $\cF$ and $\cG$ be two transverse $C^1$-foliations of the annulus $[0,1]\times \SS^1$ which are both transverse to the boundary.  We assume that $\cG$ has no compact leaves
 in the interior of the annulus.  Then there is a diffeomorphism $\theta$ of the annulus so that $\theta(\cG)$ is the horizontal foliation $\{[0,1]\times \{s\}\}_{s\in\SS^1}$ and $\theta(\cF)$ satisfies
 the following properties:
 \begin{itemize}
  \item every compact leaf of $\theta(\cF)$ is a vertical circle;
  \item every non compact leaf of $\theta(\cF)$ is transverse to the vertical circles.
 \end{itemize}
 Furthermore, $\theta$ has the same regularity as $\cF$ and $\cG$.
 Finally, if $\cG$ is  already the horizontal foliation, then $\theta$ can be chosen preserving every leaf of $\cG$ and equal to the identity map in a neighborhood of the boundary of the annulus.
\end{corollary}
\begin{proof}The unique change is that, in the last part of the proof,  we will need to use Lemma~\ref{l.analysis} in any connected component of the complement of the compact leaves of $\cF$, that is, countably many times.
For that we uses Remark~\ref{r.analysis} for choosing these diffeomorphisms arbitrarily $C^\infty$-close to identity.
\end{proof}

\subsection{Between two maximal $\cF$ and $\cG$-annuli}
The aim of this section is to  end  the proof of Theorem~\ref{t.normal-form} and therefore to end  the proof of
Theorem~\ref{thm.parallel}. We consider two transverse $C^1$ foliations $\cF$, $\cG$ on
$\TT^2$ with parallel compact leaves.

According to Proposition~\ref{p.normal-form} and~\ref{p.maximal},  there is a diffeomorphism $\theta_0$ of the torus
$\TT^2$ so that, up to replace $\cF$ and $\cG$ by $\theta_0(\cF)$ and $\theta_0(\cG)$,   there is an integer  $k>0$ for which $\cF$ and $\cG$ satisfy the following properties
\begin{itemize}
 \item both foliations $\cF$ and  $\cG$ are transverse to every vertical circle $\{\frac j{4k}\}\times\SS^1$,  for any  $j\in\ZZ/4k\ZZ$;

 \item both foliations $\cF$ and  $\cG$ have no compact leaves on the vertical annuli
 $[\frac {2i+1}{4k},\frac{2i+2}{4k}]\times \SS^1$,  for any  $i\in\ZZ/2k\ZZ$;

\item  the foliation $\cG$ coincides with the horizontal foliation on each vertical annulus
$[\frac {4i}{4k},\frac{4i+1}{4k}]\times \SS^1$,   for any  $i\in\ZZ/k\ZZ$;

 \item  there are $\frac {4i}{4k}< a_{4i} \leq b_{4i}<\frac{4i+1}{4k}$ such  that
 \begin{itemize}
 \item[--] $\{a_{4i}\}\times \SS^1$ and $\{b_{4i}\}\times \SS^1$ are compact leaves of $\cF$;
 \item[--] every compact leaf of $\cF$ in $[\frac {4i}{4k},\frac{4i+1}{4k}]\times \SS^1$ are contained in $[a_{4i},b_{4i}]\times \SS^1$;
 \item[--] $\cF$  is transverse to the vertical circles on $\big([\frac {4i}{4k}, a_{4i})\cup(b_{4i},\frac{4i+1}{4k}]\big)\times \SS^1;$
 \end{itemize}

 \item  the foliation $\cF$ coincides with the horizontal foliation on each vertical annulus
 $[\frac {4i+2}{4k},\frac{4i+3}{4k}]\times \SS^1$,    for any  $i\in\ZZ/k\ZZ$;

 \item  there are $\frac {4i+2}{4k}< a_{4i+2} \leq b_{4i+2}<\frac{4i+3}{4k}$ such that
 \begin{itemize}
 \item[--]$\{a_{4i+2}\}\times \SS^1$ and $\{b_{4i+2}\}\times \SS^1$ are compact leaves of $\cG$;
 \item[--]  every compact leaf of $\cG$ in $[\frac {4i+2}{4k},\frac{4i+3}{4k}]\times \SS^1$ are contained in
 \\$[a_{4i+2},b_{4i+2}]\times \SS^1$;
 \item[--]  $\cG$
 is transverse to the vertical circles on $\big([\frac {4i+2}{4k}, a_{4i+2})\cup(b_{4i+2},\frac{4i+3}{4k}]\big)\times \SS^1.$
 \end{itemize}
 \end{itemize}

 The following Proposition ends the proof of Theorem~\ref{t.normal-form}:
 \begin{proposition}\label{p.between} With the hypotheses and notations above, for any $i\in\ZZ/2k\ZZ$,  there is a diffeomorphism $\theta_i$ of $\TT^2$ supported on
 $( b{_{2i}},a{_{2i+2}})\times \SS^1$
 such  that for the restrictions   $\cF_i$ of $\theta_i(\cF)$ and $\cG_i$ of $\theta_i(\cG)$ to $[b{_{2i}},a{_{2i+2}}]\times \SS^1$, we have the followings:
 \begin{itemize}
 \item   the leaves of both $\cF_i$ and $\cG_i$ are transverse to every vertical circle $\{r\}\times\SS^1$,  for any $r\in [b{_{2i}},a_{2i+2}]$;
 \item   the leaves of  $\cF_i$ and $\cG_i$ satisfy one of the four possibilities below:
 \begin{enumerate}
  \item the leaves of $\cF_i$  (resp. of $\cG_i$) are not decreasing (resp. non-degenerate decreasing) on  $[b_{2i},\frac{b{_{2i}}+a_{2i+2}}{2}]\times \SS^1$  and
   are non-degenerate increasing (resp not increasing) on  $[\frac{b{_{2i}}+a{_{2i+2}}}{2},a_{2i+2}]\times \SS^1$;
  \item the leaves of $\cF_i$  (resp. of $\cG_i$) are not increasing (resp. non-degenerate increasing) on  $[b_{2i},\frac{b{_{2i}}+a{_{2i+2}}}{2}]\times \SS^1$  and
   are non-degenerate decreasing (resp. not decreasing) on  $[\frac{b{_{2i}}+a{_{2i+2}}}{2},a_{2i+2}]\times \SS^1$;
   \item the leaves of $\cG_i$  (resp. of $\cF_i$) are not decreasing (resp. non-degenerate  decreasing) on  $[b_{2i},\frac{b{_{2i}}+a{_{2i+2}}}{2}]\times \SS^1$  and
   are non-degenerate  increasing (resp. not increasing) on  $[\frac{b{_{2i}}+a{_{2i+2}}}{2},a_{2i+2}]\times \SS^1$;
  \item the leaves of $\cG_i$  (resp. of $\cF_i$) are not increasing (resp. non-degenerate  increasing) on  $[b_{2i},\frac{b{_{2i}}+a{_{2i+2}}}{2}]\times \SS^1$  and
   are non-degenerate  decreasing (resp.  not decreasing) on  $[\frac{b{_{2i}}+a{_{2i+2}}}{2},a_{2i+2}]\times \SS^1$.
 \end{enumerate}
 \end{itemize}
 \end{proposition}

   \begin{figure}[h]
\begin{center}
\def\svgwidth{0.5\columnwidth}
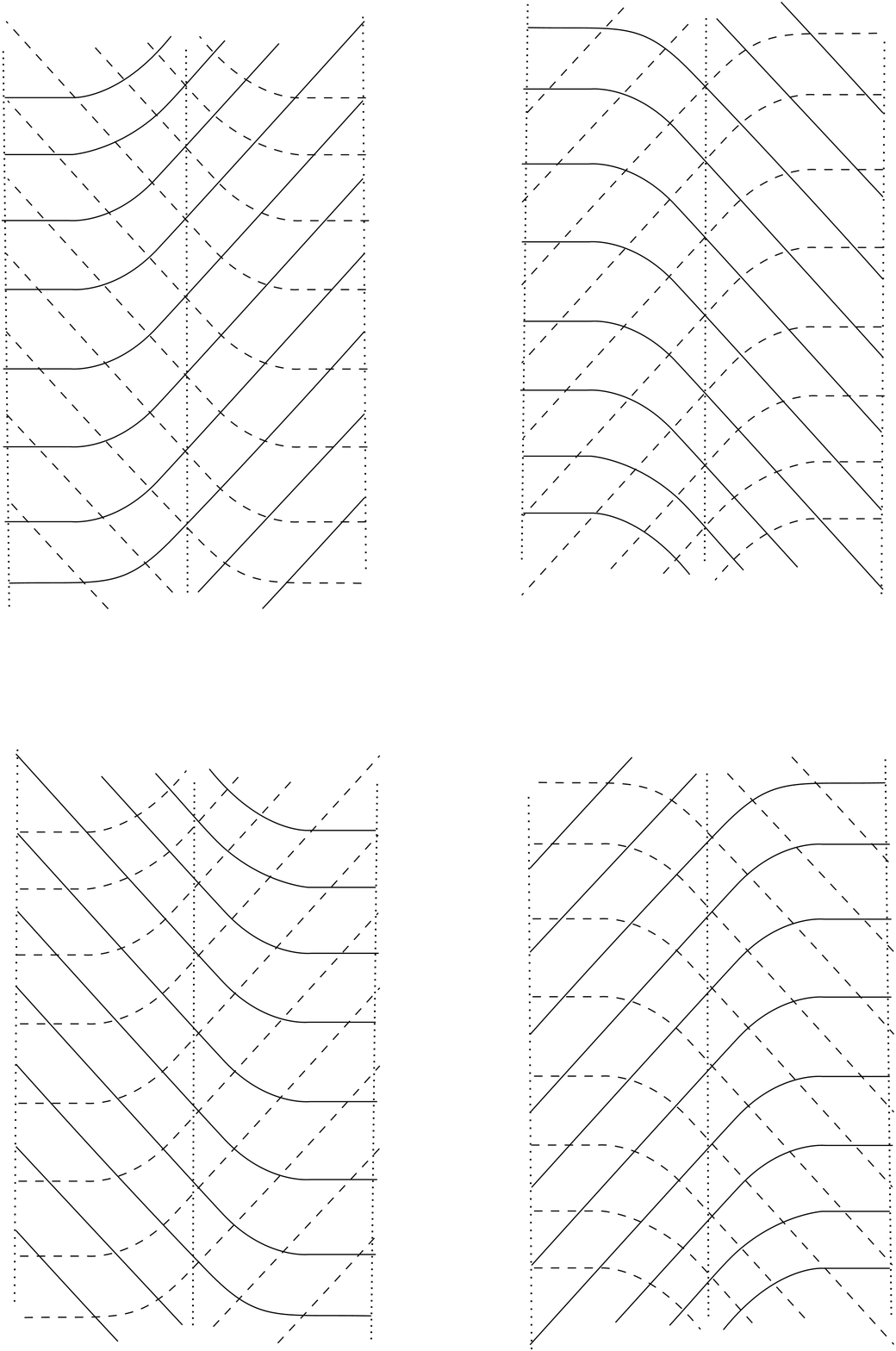
  \caption{In each figure, the real lines denote the leaves of $\cF_i$ and the dash lines denote the leaves  of $\cG_i$.}
\end{center}
\end{figure}
 We start by using Proposition~\ref{p.between} to end the proof of Theorem~\ref{t.normal-form}
 \begin{proof}[Proof of Theorem~\ref{t.normal-form}] Let  $\{\theta_\ell\}_{\ell\in\ZZ/2k\ZZ}$ be the sequence of diffeomorphisms on annuli,
 which are given by  Proposition \ref{p.between}.
 We take  four sets of points  $\{d_{4i}\}_{i\in\Z/k\Z}$, $\{c_{4i+2}\}_{i\in\Z/k\Z}$,  $\{d_{4i+2}\}_{i\in\Z/k\Z}$ and  $\{c_{4i+4}\}_{i\in\Z/k\Z}$ on $\S^1$ such that
 \begin{itemize}
 \item  $$b_{4i}<d_{4i}<\frac{b_{4i}+a_{4i+2}}2< c_{4i+2}<a_{4i+2};$$
 \item $$b_{4i+2}<d_{4i+2}<\frac{b_{4i+2}+a_{4(i+1)}}2< c_{4(i+1)}<a_{4(i+1)};$$

 \item The set $\{c_{4i+2}, c_{4i+4}, d_{4i+2},d_{4i+2}\}_{i\in\Z/k\Z}$  is disjoint from the union of the support of  all $\{\theta_\ell\}_{\ell\in\Z/2k\Z}$.
  \end{itemize}
  \begin{figure}[h]
\begin{center}
\def\svgwidth{0.8\columnwidth}
  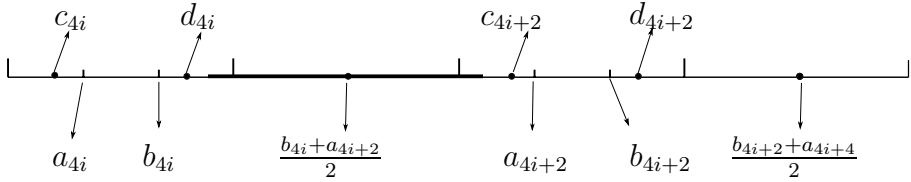
  \caption{The thick segment denotes the support of some $\theta_{\ell}$.}
\end{center}
\end{figure}

   We  choose the annuli $\{C_j\}_{j\in\ZZ/6k\ZZ}$ as  follows
 \begin{itemize}
  \item each annulus $C_{6i}$ is the vertical annulus $[c_{4i},d_{4i}]\times \SS^1$; notice that it  contains $[a_{4i},b_{4i}]\times \SS^1$ in its interior;
  \item each annulus $C_{6i+1}$ is the vertical annulus $[d_{4i},\frac12({b_{4i}+a_{4i+2}})]\times\S^1$;
  \item each annulus $C_{6i+2}$ is the vertical annulus $[\frac12({b_{4i}+a_{4i+2}}),c_{4i+2} ]\times\S^1$;
  \item each annulus $C_{6i+3}$  is the vertical annulus $[c_{4i+2},d_{4i+2}]\times \SS^1$ containing $[a_{4i+2},b_{4i+2}]\times \SS^1$ in its interior;
  \item each annulus $C_{6i+4}$ is the vertical annulus $[d_{4i+2},\frac12({b_{4i+2}+a_{4(i+1)}})]\times\S^1$;
  \item each annulus $C_{6i+5}$ is the vertical annulus $[\frac12(b_{4i+2}+a_{4(i+1)}),c_{4(i+1)}]\times\S^1$.
 \end{itemize}
 \end{proof}

It remains to prove Proposition~\ref{p.between}.

\begin{lemma}\label{l.between} Let $\cF$ and $\cG$ be two transverse foliations on $[0,1]\times \SS^1$ so that:
\begin{itemize}
 \item $\{0\}\times \SS^1$ is a compact leaf of $\cF$;
 \item $\{1\}\times \SS^1$ is a compact leaf of $\cG$;
\item $\cF$ and $\cG$ have no compact leaves in $(0,1)\times \SS^1$;
 \item  there is  a neighborhood $U_0=[0,\varepsilon_0]\times\SS^1$ of $\{0\}\times\SS^1$  on which $\cG$ coincides with the horizontal foliation and $\cF$ is transverse to the vertical circles;
 \item there is  a neighborhood $U_1=[1-\varepsilon_0,1]$ of $\{1\}\times\SS^1$  on which $\cF$ coincides with the horizontal foliation and $\cG$ is transverse to the vertical circles.
\end{itemize}

Then for any  $0<\varepsilon<\varepsilon_0$ the holonomies of $\cF$ and $\cG$ from $\Si_{0,\varepsilon}=\{\varepsilon\}\times \SS^1$ to $\Si_{1,\varepsilon}=\{1-\varepsilon\}\times \SS^1$ are well defined.
Consider the lifts $\tilde \cF$ and $\tilde \cG$ of $\cF$ and $\cG$ on the universal cover $[0,1]\times \RR$.  The holonomies $f_\varepsilon$ and $g_\varepsilon$ of
 $\tilde \cF$ and $\tilde \cG$ from
$\{\varepsilon\}\times \RR$ to $\{1-\varepsilon\}\times \RR$ are well defined.  Then for any $\varepsilon>0$ small enough one has:

 $$\left(f_\varepsilon(x)-x\right)\cdot (g_\varepsilon(x)-x)<0, \mbox{ for every }x\in\RR.$$
\end{lemma}
\begin{proof}On $U_0\setminus\{0\}\times\SS^1$,  the foliation $\cF$ is transverse to the horizontal segments and to the vertical circles.  Therefore its leaves are either non-degenerate  increasing or non-degenerate  decreasing curves.
Let us assume that they are non-degenerate increasing (the other case is similar).

Notice that, on $[\varepsilon_0,1-\varepsilon_0]\times \SS^1$, the foliations $\cF$ and $\cG$ are transverse to   the boundary and are transverse to each other. We orient $\cF$ and $\cG$ from $\{\varepsilon_0\}\times\SS^1$ to
$\{1-\varepsilon_0\}\times \SS^1$. As $\cF$ is increasing along $\{\varepsilon_0\}\times\SS^1$ and horizontal along $\{1-\varepsilon_0\}\times\SS^1$, and as $\cG$ is horizontal along
$\{\varepsilon_0\}\times\SS^1$, one gets that $\cG$ is decreasing along $\{\varepsilon_0\}\times\SS^1$.  Thus the leaves of $\cG$ are decreasing curves on $U_1\setminus\{1-\varepsilon_0\}\times \SS^1$.

Let us denote by
            $$f_{\varepsilon,0}\colon \{\varepsilon\}\times \RR\to \{\varepsilon_0\}\times \RR$$
            $$f_{1,\varepsilon}\colon \{1-\varepsilon_0\}\times \RR\to \{1-\varepsilon\}\times \RR$$
             and
$$g_{\varepsilon,0}\colon \{\varepsilon\}\times \RR\to \{\varepsilon_0\}\times \RR$$
$$g_{1,\varepsilon}\colon \{1-\varepsilon_0\}\times \RR\to \{1-\varepsilon\}\times \RR$$
the holonomies of $\tilde \cF$ and $\tilde \cG$
on the corresponding transversals. We consider them as diffeomorphisms of $\RR$ (that is we forget the horizontal coordinate).

Then $g_{\varepsilon,0}=f_{1,\varepsilon}$ are equal to the identity map as they are horizontal foliations in the corresponding regions.

Thus one gets that $$f_\varepsilon= f_{\varepsilon_0}\circ f_{\varepsilon,0} \mbox{ and }g_\varepsilon=g_{1, \varepsilon}\circ g_{\varepsilon_0}.$$

Now Lemma~\ref{l.between} follows directly from the following claim:
\begin{claim} $f_{\varepsilon,0}(x)-x$ and $g_{1,\varepsilon}(x)-x$ converge uniformly to $+\infty$ and $-\infty$, respectively, as $\varepsilon$ tends to $0$.
\end{claim}
The claim follows directly from the fact that the leaves of $\tilde \cF$ (resp. $\tilde \cG$) are non-degenerate
increasing (resp. non-degenerate decreasing) curves asymptotic to the vertical line $\{0\}\times\RR$ (resp. $\{1\}\times\RR$) according to the negative orientation (resp. positive orientation).
\end{proof}

One ends the proof of Proposition~\ref{p.between} by proving:
\begin{lemma}\label{l.monotonous} Let $\cF$ and $\cG$ be two transverse $C^1$ foliations on the annulus $[0,1]\times \SS^1$
which are transverse to the boundary and do not have any compact leaf in the interior.
We denote by $\tilde \cF$ and $\tilde \cG$ the lifts of $\cF$ and $\cG$ to $[0,1]\times\RR$.
Under that hypotheses,  the holonomies of
$\tilde \cF$ and $\tilde \cG$ from $\{0\}\times \RR$ to $\{1\}\times \RR$ are well defined and we denote them as $f$ and $g$,
respectively (and we consider them as diffeomorphisms of $\RR$).
Assume that:
\begin{itemize}
 \item[--] the foliation $\cG$ (resp. $\cF$) coincides with the horizontal foliation on a neighborhood of $\{0\}\times \SS^1$ (resp. $\{1\}\times \SS^1$);
 \item[--]  for every $x\in\RR$,  one has $f(x)>x$ and $g(x)<x$.
\end{itemize}
Then there is a diffeomorphism $\theta$  of $[0,1]\times \SS^1$, equal to the identity map on a neighborhood of the boundary, and isotopic to the identity relative to the boundary, and so that (denoting by
$\tilde\cF_\theta$ and $\tilde \cG_\theta$ the lifts of $\theta(\cF)$ and $\theta(\cG)$ to $[0,1]\times \RR$):
\begin{itemize}
\item the leaves of $\theta(\cF)$ and of $\theta(\cG)$ are transverse to the vertical circles;
 \item the leaves $\tilde\cF_\theta$ are non-degenerate increasing  on $[0,\frac12]\times\RR$ and are not decreasing on $[\frac 12,1]\times \RR$;
 \item the leaves of $\tilde\cG_\theta$ are not increasing  on $[0,\frac12]\times\RR$ and are non-degenerate decreasing on $[\frac 12,1]\times \RR$;
\end{itemize}
\end{lemma}
\begin{proof}[Sketch of proof] We just need to choose a pair of transverse $C^1$ foliations $\cF_0$ and $\cG_0$ so that, denoting  by $\tilde \cF_0$ and $\tilde\cG_0$ their lifts on $[0,1]\times \RR$, one has:
\begin{itemize}
 \item $\cF_0$ and $\cG_0$ are transverse to the vertical foliation,
 \item $\cF_0$ and $\cG_0$ coincide with $\cF$ and $\cG$ in a neighborhood of the boundary
 \item the holonomies of $\tilde \cF_0$ and $\tilde \cG_0$ from $\{0\}\times \RR$ to $\{1\}\times\RR$ are $f$ and $g$, respectively,
 \item the leaves $\tilde\cF_0$ are non-degenerate increasing  on $[0,\frac12]\times\RR$ and are not decreasing on $[\frac 12,1]\times \RR$;
 \item the leaves of $\tilde\cG_0$ are not increasing  on $[0,\frac12]\times\RR$ and are non-degenerate decreasing  on $[\frac 12,1]\times \RR$;
\end{itemize}
The fact that we can choose such a pair of foliations is similar to the proof of Proposition~\ref{p.transverse to horizontal}.

Then the pair $(\cF,\cG)$ is conjugated to $(\cF_0,\cG_0)$ by a unique diffeomorphism equal to the identity map  in a neighborhood of the boundary.
The lift $\tilde \theta$ on $[0,1]\times \RR$ of the announced diffeomorphism $\theta$ is build as follows:
consider a point $p\in[0,1]\times \RR$ and let $q_F(p)$ and $q_G(p)$ be the intersections with $\{0\}\times \RR$ of the leaves  $\tilde \cF_p$ and $\tilde \cG_p$ through $p$.
The transversality of $\tilde \cF$ and $\tilde \cG$
implies that  $q_{_F}(p)$ is below $q_{_G}(p)$  and $f(q_{_F}(p))$ is over $g(q_{_G}(p))$. As $\tilde\cF_0$ and $\tilde \cG_0$ have the same holonomies as $\tilde\cF$ and $\tilde \cG$,   one gets that
the leaves of $\tilde \cF_0$ and of $\tilde\cG_0$ through $q_{_F}(p)$ and $q_{_G}(p)$ have a unique intersection point that we denote by $\tilde \theta(p)$.
\end{proof}

\section{Dehn twist, transverse foliations and partially hyperbolic diffeomorphism}
The aim of this section is the proof of Proposition~\ref{p.foliations} and of Theorem~\ref{t.Anosov}.

\subsection{Transverse foliations on $3$-manifolds, and the proof of Proposition~\ref{p.foliations}}
Let $M $ be a closed 3-manifold and $\cF$ and $\cG$ be transverse codimension one  foliations of class $C^1$ on $M$.
Thus $\cF$ and $\cG$ intersect each other along a $C^1$ foliation $\cE$ of dimension $1$.
We assume that there is a torus $T$ embedded in $M$ such that $\cE$ is transverse to $T$, and we denote by
$\cF_T$ and $\cG_T$ the $1$-dimensional $C^1$ foliations on $T$ obtained by intersecting  $T$ with
$\cF$ and $\cG$, respectively.

There is a collar neighborhood $U$ of $T$ and an orientation preserving diffeomorphism
$\psi\colon U \to T\times [0,1]$ inducing the identity map from $T$ to $T\times \{0\}$,  so
that $\theta(\cE)$ is the trivial foliation $\{\{p\}\times[0,1]\}_{p\in T}$.  Then $\theta(\cF)$
and $\theta(\cG)$ are the product foliations of $\cF_T\times[0,1]$ and $\cG_T\times[0,1]$,
respectively (meaning that their leaves are the product by $[0,1]$ of  the leaves of $\cF_T$ and $\cG_T$, respectively).

Let $u$ be an element of $G_{\cF_T,\cG_T}\subset \pi_1(T)$.  By definition of $G_{\cF_T,\cG_T}$,  there is a loop
$\{\varphi_t\}_{t\in[0,1]}$ of $C^1$ diffeomorphisms of $T$ so that $\varphi_0=\varphi_1$ is the identity map, $\varphi_t(\cF_T)$
is transverse to $\cG_T$
and for any $p\in T$,  the loop $\{ \varphi_t(p)\}_{ t\in[0,1]}$  belongs to the homotopy class of $u$.

We consider the diffeomorphism $\Phi$ on $T\times[0,1]$ defined by $(p,t)\mapsto (\varphi_{\alpha(t)}(p),t)$, where $\alpha\colon [0,1]\to[0,1]$ is
a smooth function equal to $0$ in a neighborhood of $0$ and to $1$ in a neighborhood of $1$.
Then $\Phi$ is a Dehn twist directed by $u$ and $\Phi$ is the identity map in a neighborhood of the boundary of $T\times [0,1]$.

Consider   $\Phi(\cF_T\times [0,1])$.  It is a foliation transverse to any torus $T\times\{t\}$ and it induces $\varphi_{\alpha(t)}(\cF_T)$ on $T\times \{t\}$.
Therefore it is transverse to $\cG_T$.

This proves that $\Phi(\cF_T\times [0,1])$ is transverse to the foliation $\cG_T\times[0,1]$.

Now the announced Dehn twist on $M$ directed by $u$ is the diffeomorphism $\psi$  with support in $U$ and whose restriction to $U$ is $\theta^{-1}\circ\Phi\circ\theta$.
By construction $\theta(\cF)$ is transverse to $\cG$, ending the proof.

\subsection{Anosov flows, Dehn twist and partially hyperbolic diffeomorphisms}
Let $X$ be a non-transitive Anosov vector field of class at least $C^2$ on a closed 3-manifold $M$ and we denote by $X_t$ the flow generated by $X$.
According to Proposition~\ref{p.Lyapunov function}, any family of transverse tori on which $X$ has no return,  are contained in a regular level of a smooth Lyapunov function.

Let $L(x):M\mapsto \R$ be a smooth  Lyapunov function of the flow $X_t$, and let $c$ be a regular value of $L$.
Thus each connected component of $L^{-1}(c)$ is a torus transverse to $X$.

Let $T_1,\cdots, T_k$ be the  disjoint transverse tori such that $$\cup_{i=1}^k T_i=L^{-1}(c).$$

Consider the set $M^{r}=L^{-1}(c,+\infty)$ and $M^{a}=L^{-1}(-\infty,c)$. Then $M^{r}$ and $M^{a}$ are two disjoint open subsets of $M$ and share the same boundary
$\cup_{i=1}^k T_i$. Since $L(x)$ is strictly decreasing along the positive  orbits of the points in the wandering domain,
one gets  that $M^{a}$ and $M^{r}$ are attracting and repelling regions of the vector field $X$. We denote by $\cA$ and $\cR$, respectively, the maximal invariant sets
of $X$ in $M^a$ and $M^r$. Thus $\cA$ is a hyperbolic (not necessarily transitive)  attractor  and $\cR$ is a  hyperbolic (not necessarily transitive) repeller for $X$.

 By ~\cite[Corollary 4]{HP}, the center stable foliation $\calF_X^{cs}$ and center unstable foliation $\calF_X^{cu}$ of the Anosov flow $X_t$ are $C^1$ foliations.
  For each $i=1,\cdots,k$,  we denote by $\cF^s_i$ and $\cF^{u}_i$  the $C^1$ foliation induced by $\calF_X^{cs}$ and $\cF_X^{cu}$ on $T_i$ respectively.

 As $X$ has no return on  $\bigcup_iT_i$,   the sets $\{X_t(T_1)\}_{t\in\R},\cdots,$ $ \{X_t(T_k)\}_{t\in\R}$ are pairwise disjoint embeddings of $T_i\times\RR$ into $M$.
As a consequence, for any integer $N$,  the sets $\{X_{t}(T_1)\}_{t\in [0,N]},\cdots, \{X_{t}(T_k)\}_{t\in [0,N]}$ are pairwise disjoint  and diffeomorphic to $\T^2\times[0,N]$.

For  each $i$,  we define the diffeomorphism
$$\psi_{i,N}:\{X_{t}(T_i)\}_{t\in [0,N]}\mapsto T_i\times [0, 1]$$
by $(X_t(p))\mapsto (p, t/N)$, for any $p\in T_i$ and $t\in[0,N]$.
Thus $D\psi_{i,N}(X)=\frac 1N \frac{\partial}{\partial s}$  is tangent to the vertical segments $\{p\}\times[0,1]$, for any $p\in T_i$.

We fix a smooth function $\alpha(s):[0,1]\mapsto[0,1]$ such that $\alpha(s)$ is a non-decreasing function on $[0,1]$,  equals to $0$ in a small neighborhood of $0$,
 and equals to $1$ in a small neighborhood of $1$.
For each $i$, the group $G_i=G_{\cF^s_i,\cF^u_i}$ is the subgroup of $\pi_1(\T^2)$ associated to the pair of transverse foliations $(\cF^s_i,\cF^{u}_i)$ by Definition \ref{d.G}.
 Given an element $u_i\in G_i $,  let $\{\varphi_{_t}\}_{t\in[0,1]}$ be the loop in $\diff^1(T_i)$ associated to $u_i$ by Theorem~\ref{t.transverse}.

Consider the map  $\Phi_i : T_i\times [0,1]\mapsto T_i\times [0,1]$  defined as
$$ (x,s)\mapsto (\varphi_{_{\alpha(s)}}(x),s).$$

Hence, the map  $\Psi_{i,N}=\psi_{i,N}^{-1}\circ\Phi_i\circ\psi_{i,N}$ is a Dehn twist directed by $u_i$. Notice that $\Psi_{i,N}$ can be $C^1$-smoothly extended on
the whole manifold $M$ to be the identity map outside $X_{t}(T_i)$.

 The main part of Theorem \ref{t.Anosov} is directly implied by the following theorem:
\begin{theorem}\label{c.compose Dehn twist}With the notation above, when $N$ is chosen large enough, the  diffeomorphism $\Psi_{k,N}\circ\cdots\circ\Psi_{1,N}\circ X_N$ is absolutely partially hyperbolic.
\end{theorem}
\begin{proof}
We denote $$\Psi_N=\Psi_{k,N}\circ\cdots\circ \Psi_{1,N}.$$

Then $\Psi_N\circ X_N$  coincides with $X_N$ on the attracting region $M^a$.
Thus $\cA$ is the maximal invariant set of $\Psi_N\circ X_N$ in $M^a$ and is an absolute partially hyperbolic attractor.
Furthermore the center stable bundle and the strong stable bundle on $\cA$ admit   unique continuous and ($\Psi_N\circ X_N$)-invariant
extensions $E^{cs}_\cA$ and $E^s_\cA$, respectively, on $M^a$  which coincide with the tangent bundles of the center stable and strong stable foliations  $\cF^{cs}_X$ and $\cF^{s}_X$ of the vector field $X$.

In the same way, $(\Psi_N\circ X_N)^{-1}$ coincides with $X_{-N}$ on the repelling region  $M^r$. Thus $\cR$ is still an absolute partially hyperbolic repeller of $\Psi_N\circ X_N$ and its
center unstable and strong unstable bundles admit  unique continuous and ($\Psi_N\circ X_N$)-invariant extensions $E^{cu}_\cR$ and $E^u_\cR$ on $M^r$ which coincide with, respectively,
the tangent bundles of $\cF^{cu}_X$ and $\cF^{u}_X$.

Notice that the center unstable and strong unstable bundles  $E^{cu}_\cR$ and $E^u_\cR$, of the repeller $\cR$ for $\Psi_N\circ X_N$ extend in a unique way on $M\setminus \cA$, just by
pushing by the dynamics of $\Psi_N\circ X_N$.

Thus the bundles $E^{cu}_\cR$, $E^u_\cR$, $E^{cs}_\cA$ and $E^s_\cA$  coincide with the tangent bundles of the foliations  $\Psi_N(\cF^{cu}_X)$, $\Psi_N(\cF^u_X)$, $\cF^{cs}_X$, and $\cF^s_X$respectively,
on the fundamental domain $\bigcup_iX_{[0,N]}(T_i)$.


One can easily check the following classical result:
\begin{lemma}\label{l.criteria} $\Psi_N\circ X_{N}$ is absolutely partially hyperbolic if and only if
$$\Psi_N(\cF_X^{u})\pitchfork \cF^{cs}_X\textrm{ and } \Psi_N(\cF^{cu}_{X})\pitchfork \cF^{s}_X.$$
\end{lemma}

Notice that $\{X_t(T_1)\}_{t\in\R},\cdots, \{X_t(T_k)\}_{t\in\R}$ are pairwise disjoint, the same argument of  Lemma 6.2 in \cite{BPP} gives the following:
 \begin{lemma}\label{l.limit transverse}With the notation above, we have that for each $i=1,\cdots,k$,
 $$\lim_{N\rightarrow +\infty}\psi_{i,N}{(\cF^{uu}_X)}=\{\cF^{u}_i\}\times \{s\}\textrm{ and }\lim_{N\rightarrow +\infty}\psi_{i,N}{(\cF^{ss}_X)}=\{\cF^{s}_i\}\times \{s\}$$
 uniformly in the $C^1$-topology.
 \end{lemma}
 As a consequence of Lemma \ref{l.limit transverse}, when $N$ is chosen large,  for each $i=1,\cdots,k$, we have that
 $$\Phi_i(\psi_{i,N}{(\cF^{uu}_X)})\pitchfork  \cF^s_i\times[0,1] \textrm{ and } \Phi_i(\psi_{i,N}{(\cF^{ss}_X)})\pitchfork  \cF^u_i\times[0,1].$$
 Now Theorem \ref{c.compose Dehn twist} follows directly from Lemma \ref{l.criteria}.

 \end{proof}

Now, we end the proof of Theorem~\ref{t.Anosov} by proving that the (absolute) partially hyperbolic diffeomorphism $f=\Psi_N\circ X_N$ is robustly dynamically coherent and plaque expansive.
We denote by $E^c_f$ the center bundle of $f$.

Recall that $f$ coincides with $X_N$ on the repelling region $X_{-N}(M^{r})$ and on the attracting region $M^{a}$. Just as Lemma 9.1 in \cite{BPP}, we have that:
 \begin{lemma}\label{l.neutral} There exists a constant $C>1$ such that for any unit vector $v\in E^c_f$, we have the following:
 $$\frac{1}{C}\leq \norm{Df^n(v)}\leq C, \textrm{ for any integer $n\in\Z$}.$$
 \end{lemma}

As a consequence of Lemma \ref{l.neutral}, we have that $f$ is \emph{Lyapunov stable} and \emph{Lyapunov unstable} in  the directions $E^{cs}_f$ and $E^{cu}_f$ respectively.

To show the dynamically coherent and plaque expansive properties, we follow  the same  argument in \cite[Theorem 9.4]{BPP}:
\begin{itemize}
\item[--]According to \cite[Theorem 7.5]{HHU1},  $f$ is \emph{dynamically coherent},  and center stable foliation $\cW^{cs}_f$ and center unstable foliation$\cW^{cu}_f$ are \emph{plaque expansive};
\item[--] By \cite{HPS},    the center stable foliation $\cW^{cs}_f$ and center unstable foliation  $\cW^{cu}_f$ are \emph{structurally stable}, proving that $f$ is robustly dynamically coherent.
\end{itemize}

\bibliographystyle{plain}

\vskip 2mm

\noindent Christian Bonatti,

\noindent {\small Institut de Math\'ematiques de Bourgogne\\
UMR 5584 du CNRS}

\noindent {\small Universit\'e de Bourgogne, 21004 Dijon, FRANCE}

\noindent {\footnotesize{E-mail : bonatti@u-bourgogne.fr}}

\vskip 2mm
\noindent Jinhua Zhang,

\noindent{\small School of Mathematical Sciences\\
Peking University, Beijing 100871, China}

\noindent {\footnotesize{E-mail :zjh200889@gmail.com }}\\
\noindent and\\
\noindent {\small Institut de Math\'ematiques de Bourgogne\\
UMR 5584 du CNRS}

\noindent {\small Universit\'e de Bourgogne, 21004 Dijon, FRANCE}\\
\noindent {\footnotesize{E-mail : jinhua.zhang@u-bourgogne.fr}}

\end{document}